\documentclass[11pt,twoside]{article}

\usepackage[square,numbers]{natbib}

\iffalse 
\usepackage[autostyle]{csquotes}
\usepackage[
    backend=biber,
    style=numeric,
    firstinits = true,
    isnb=false,
    url=false, 
    doi=true,
    eprint=true
]{biblatex}

\renewbibmacro{in:}{
  \ifentrytype{article}{}{
  \printtext{\bibstring{in}\intitlepunct}}}

\DeclareFieldFormat[article,unpublished,book]{pages}{#1}

\DeclareFieldFormat[article]{title}{#1}
\DeclareFieldFormat[unpublished]{title}{#1}
\DeclareFieldFormat[misc]{title}{#1}
\fi

\newcommand{\aux}{\mathrm{aux}}

\usepackage{lipsum}

\usepackage{microtype} 

\usepackage{lmodern}
\usepackage[sans]{dsfont}
\usepackage[latin1]{inputenc}
\usepackage{color}
\usepackage{delimdelim}

\usepackage[hmarginratio=1:1,top=32mm,columnsep=20pt]{geometry} 
\usepackage{booktabs} 
\usepackage{hyperref} 

\usepackage{paralist} 

\usepackage{abstract}

\newcommand{\ep}{\epsilon}
\newcommand{\m}{\mbox{d}}
\newcommand{\ov}{\overline}

\definecolor{blue1}{rgb}{0,0,0}
\definecolor{G1}{rgb}{0.3,0.7,0.0}

\newcommand*{\doi}[1]{\href{http://dx.doi.org/#1}{doi: #1}}

\usepackage{titlesec}

\usepackage{fancyhdr} 
\usepackage{amsmath}
\usepackage{amssymb}
\usepackage{amsthm}
\usepackage{subfigure}
\usepackage{textcomp}
\usepackage{listings}
\usepackage{setspace}
\usepackage{graphicx}
\usepackage{epstopdf}
\usepackage{eurosym}
\usepackage{inputenc}
\usepackage{inputenc}
\usepackage{appendix}
\usepackage{tabularx}
\usepackage{array}

\usepackage{makeidx}

\usepackage{etoolbox}
\usepackage{mathtools}
\usepackage{empheq}
\usepackage{mathrsfs}
\usepackage{marginnote}

\DeclareGraphicsExtensions{.eps}
\DeclareGraphicsExtensions{.png}

\theoremstyle{plain}
\newtheorem{theorem}{Theorem}[section]

\newtheorem{lemma}[theorem]{Lemma}

\newtheorem{prop}[theorem]{Proposition}

\newtheorem*{assumg*}{Assumption~(G)}

\newtheorem*{assumng*}{Assumption~(NG)}
\newtheorem*{assumng1*}{Assumption~(NG-bis)}

\theoremstyle{definition}

\newtheorem{rem}[theorem]{Remark}
\newtheorem{ass}[theorem]{Assumption}
\hypersetup{citecolor=black}

\label{mathrefs}

\title{\vspace{-15mm}\fontsize{24pt}{10pt} \LARGE{\scshape{From weakly interacting particles to a regularised Dean--Kawasaki model}}} %
\small{
\author{
\textsc{Federico Cornalba\footnote{e-mail: {\ttfamily F.Cornalba@bath.ac.uk}},\hspace{0.2 pc} Tony Shardlow\footnote{e-mail: {\ttfamily T.Shardlow@bath.ac.uk}},\hspace{0.2 pc} Johannes Zimmer\footnote{e-mail: {\ttfamily J.Zimmer@bath.ac.uk}}}\vspace{0.5 pc}\\ \footnotesize{Department of Mathematical Sciences, University of Bath, Bath, BA2 7AY, United Kingdom}
\vspace{-5mm}
}}
\date{}

\begin{document}

\maketitle 
\thispagestyle{empty}

\pagestyle{fancy} 

\renewenvironment{abstract}
 {\small
  \begin{center}
  \bfseries \abstractname\vspace{-0.0 pc}\vspace{0pt}
  \end{center}
  \list{}{
    \setlength{\leftmargin}{10mm}
    \setlength{\rightmargin}{\leftmargin}
  }
  \item\relax}
 {\endlist}

 \newenvironment{thanks}
 {  \list{}{
    \setlength{\leftmargin}{0mm}
    \setlength{\rightmargin}{\leftmargin}
  }
  \item\relax}
 {\endlist}

\vspace{-1 pc}
\begin{abstract} 
The evolution of finitely many particles obeying Langevin dynamics is described by Dean--Kawasaki equations, a class of stochastic equations featuring a non-Lipschitz multiplicative noise in divergence form. We derive a regularised Dean--Kawasaki model based on second order Langevin dynamics by analysing a system of particles interacting via a pairwise potential. 
Key tools of our analysis are the propagation of chaos and Simon's compactness criterion. The model we obtain is a small-noise stochastic perturbation of the undamped McKean--Vlasov equation.
We also provide a high-probability result for existence and uniqueness for our model. 

  

\end{abstract}

\section{Introduction}

The Dean--Kawasaki model~\cite{Dean1996a,Kawasaki1998a} describes the evolution of a system of finitely many particles obeying Langevin dynamics. A key feature of the particle system is the stochastic independence of the forcing terms driving the particles. The particles themselves, on the other hand, might be independent~\cite{Lutsko2012a} or interact through a
potential~\cite{Dean1996a}: in this work, we focus on the latter case.

In its simplest form, the Dean--Kawasaki model reads
\begin{align}
  \label{i:310}
  \partial_t\rho=\nabla\cdot\left(\rho\,
  \nabla\frac{\delta F(\rho)}{\delta \rho}\right)+\nabla\cdot\left(\sigma\sqrt{\rho}\,\xi\right),
\end{align} 
with $\sigma\in\mathbb{R}$, where $\rho$ is the particle density, $F$ is an energy functional, and $\xi$ is a space-time white
noise. 
The model~\eqref{i:310} may be obtained from either a first-order Langevin equation~\cite{Dean1996a}, or from second-order
Langevin dynamics in an overdamped limit~\cite{Lutsko2012a}. 

Equations such as~\eqref{i:310} pose a challenge for existence theory, in particular due to the multiplicative structure of the noise in divergence form and to its square-root
coefficient function. 
The latter is related to the independence of the forcing terms driving the
particles~\cite{Dean1996a,Lutsko2012a}. Consequently, well-posedness for~\eqref{i:310} is an open question, with the exception of
the purely diffusive case~\cite{vonRenesse18}. More specifically, for the deterministic drift being $\frac{N}{2}\Delta$, where
$N>0$, equation~\eqref{i:310} admits a unique trivial (atomic) solution only if $N\in\mathbb{N}$, and has no solutions if
$N\notin \mathbb{N}$. This striking result indicates how subtle the analysis of equations of this kind is. 

In order to obtain
non-trivial solutions to~\eqref{i:310}, different approaches have been developed in recent years. One approach is to correct the
drift~\cite{Renesse2009a,Andres2010a,Konarovskyi2018a,Konarovskyi2017a}, another one is to regularise the
equation~\cite{Fehrman2017a,Mariani2010a}. For a regularised undamped equivalent of~\eqref{i:310}, corresponding to a regularised
stochastic wave equation in the density/momentum density pair $(\rho,j)$, a result of existence and uniqueness is found
in~\cite{Cornalba2018aTR}; that model, here referred to as the \emph{regularised Dean--Kawasaki} model, is derived from independent
particles.
The key regularisation chosen in~\cite{Cornalba2018aTR} is a representation of particles by Gaussians, rather than their limiting
Dirac measures. The main contributions of this work is to extend this idea to some important systems of interacting particles. 
Specifically, we derive and analyse a
regularised Dean--Kawasaki model set in the undamped regime, as in~\cite{Cornalba2018aTR}, but describing the evolution of a
system of finitely many weakly interacting particles governed by undamped McKean--Vlasov dynamics, see for
example~\cite{Duong2013a,Bolley10,Monmarche2017a}.  

Throughout the paper, we rely on some methodology found in~\cite{Cornalba2018aTR}.  However, the interaction of the particles also
requires various new approaches. Specifically, in contrast to~\cite{Cornalba2018aTR}, we employ propagation of chaos
techniques~\cite{Malrieu01} and Simon's compactness criterion~\cite{Simon1987a} to overcome the difficulties posed by
stochastically dependent particles. In addition, as the resulting model is superlinear (as specified below), we also need to localise the solutions using suitable stopping times. 
More details are provided
in Subsection~\ref{ss:2} below.

\subsection{Weakly interacting particles on a one-dimensional torus}\label{ss:1}

The system studied here consists of $N$ interacting particles on the one-dimensional flat torus of length one, denoted by $\mathbb{T}$. Each particle $i\in\{1,\dots,N\}$ is described in terms of position and velocity $(q_i,p_i)\in\mathbb{T}\times\mathbb{R}$. 
The system obeys the following undamped Langevin dynamics on a probability space $(\Omega,\mathcal{F},\mathbb{P})$, 
\begin{equation}
  \label{i:1}
  \left\{
    \begin{array}{l}
      \displaystyle \dot{q}_i= p_i,  \vspace{0.2 pc}\\
      \displaystyle \dot{p}_i= -\gamma p_i-\frac{1}{N}\sum_{j=1}^{N}{W'(q_i-q_j)}+\sigma\,\dot{\beta}_i,\qquad i=1,\dots,N,
    \end{array}
  \right.
\end{equation}
where $\{\beta_i\}_{i=1}^{N}$ are independent Brownian motions, the interaction potential $W$ is periodic and smooth, say $W\in C^2(\mathbb{T})$, the initial conditions $\{(q_{i,0},p_{i,0})\}_{i=1}^{N}$ are independent and identically distributed, and $\sigma$ and $\gamma$ are positive constants. The dissipative term $-\gamma p_i$ is a frictional drag, balancing the fluctuating Brownian term $\sigma\dot{\beta_i}$. The particles $\{(q_i,p_i)\}_{i=1}^{N}$ are exchangeable, but not necessarily independent. 
\begin{rem}
Throughout this work, diacritical dots ($\,\dot{}\,$) are used to indicate time differentiation of finite or infinite dimensional It\^o processes (e.g., see \eqref{i:1}).
\end{rem}

In order to study~\eqref{i:1}, we introduce an auxiliary Langevin system of particles $\{(\ov{q}_i,\ov{p}_i)\}_{i=1}^{N}$ obeying
\begin{equation}
  \label{i:2}
  \left\{
    \begin{array}{l}
      \displaystyle \dot{\ov{q}}_i= \ov{p}_i,  \vspace{0.2 pc}\\
          \displaystyle \dot{\ov{p}}_i= -\gamma \ov{p}_i-W'\ast\mu_t(\ov{q}_i)+\sigma\,\dot{\beta}_i,\qquad i=1,\dots,N,
    \end{array}
  \right.
\end{equation}
where $\ast$ denotes the convolution operator on $\mathbb{T}$, $\mu_t$ denotes the law
of $\ov{q}_i(t)$, and the Brownian motions and the initial conditions coincide $\mathbb{P}$-a.s. with their respective counterparts in~\eqref{i:1}. As a result of these assumptions, the particles $\{(\ov{q}_i,\ov{p}_i)\}_{i=1}^{N}$ are clearly independent. System~\eqref{i:2} is associated with the Vlasov--Fokker--Planck equation
\begin{align}\label{i:312}
\frac{\partial f_t}{\partial t}+p\frac{\partial f_t}{\partial q}-W'\ast \rho[f_t](q)\frac{\partial f_t}{\partial p}=\frac{\sigma^2}{2}\Delta_pf_t+\frac{\partial (\gamma pf_t)}{\partial p}
\end{align}
in the probability density function $f_{t}(q,p)\colon [0,T]\times\mathbb{T}\times\mathbb{R}\rightarrow [0,\infty)$, where $\rho[f_t](q)=\int_{\mathbb{R}}{f_t(q,p)\m p}$; see~\citep{Bolley10,Villani2009b}.

\subsection{Outline of the paper}\label{ss:2}

We derive and analyse a regularised Dean--Kawasaki model in the undamped regime, based on the interacting particle system~\eqref{i:1}. A portion of our analysis is based on~\cite{Cornalba2018aTR}, and the relevant methodological novelties are sketched and put into context below. 

Section \ref{s:2} contains some auxiliary results. Subsection \ref{ss:2.1} establishes a propagation of chaos result (Proposition \ref{ip:1}) linking~\eqref{i:1} and~\eqref{i:2}, using ideas from~\citep{McKean67,Malrieu01}. This sort of result, which is not required in~\cite{Cornalba2018aTR}, is here needed to compare the system of interest~\eqref{i:1} to the more tractable system of independent particles~\eqref{i:2}. Specific aspects of the latter system's regularity, and in particular of the regularity of solutions to~\eqref{i:312}, are studied in Proposition \ref{ip:2} in Subsection \ref{ss:2.2}; there, we explain the reason for choosing $\mathbb{T}$ (rather than $\mathbb{R}$ as in~\cite{Cornalba2018aTR}) as the spatial domain. Subsection \ref{ss:2.3} relies on Propositions \ref{ip:1} and \ref{ip:2} to establish Proposition \ref{ip:4}:
 for $\ep>0$, this result provides $\ep$-independent uniform estimates for certain Sobolev-space norms applied to the \emph{regularised} densities
 \begin{align}
\rho_{\ep}(x,t) & := \frac{1}{N}\sum_{i=1}^{N}{w_{\ep}(x-q_i(t))},\quad j_{\ep}(x,t):=\frac{1}{N}\sum_{i=1}^{N}{p_i(t)w_{\ep}(x-q_i(t))},\label{i:311a}\\
j_{2,\ep}(x,t) & :=\frac{1}{N}\sum_{i=1}^{N}{p^2_i(t)w'_{\ep}(x-q_i(t))}.\label{i:311b}
\end{align}
Above, $(x,t)\in\mathbb{T}\times[0,T]$, while $w_{\ep}$ is the periodic von Mises distribution~\citep{Forbes2011a} on $\mathbb{T}$ with location parameter $\mu:=0$ and concentration parameter $\kappa:=\ep^{-2}$, namely,
\begin{align}\label{i:18}
w_{\ep}(x)
:=Z_{\ep}^{-1}e^{-\frac{\sin^2(x/2)}{\ep^2/2}},\qquad Z_{\ep}:=\int_{\mathbb{T}}{e^{-\frac{\sin^2(x/2)}{\ep^2/2}}\m x}.
\end{align}
The quantities in~\eqref{i:311a} are the \emph{regularised} empirical density and momentum density for~\eqref{i:1}, and will be the building block of our final model; as for~\eqref{i:311b}, this is a relevant auxiliary quantity emerging from the analysis of~\eqref{i:311a}.

The kernel $w_\ep$ is introduced for smoothing and regularisation purposes. 
More precisely, we work with the quantities~\eqref{i:311a}--\eqref{i:311b} rather than their atomic counterparts defined by a replacement of $w_{\ep}$ with Dirac delta functions centred on the particles; this is a key aspect of our approach, as it allows us to use standard tools from stochastic analysis and work with smooth functions. We refer to~\cite[Section 1]{Cornalba2018aTR} for a similar discussion. 
The kernel $w_\ep$, which recovers a Dirac delta as $\ep\rightarrow 0$, is the toroidal equivalent of a Gaussian distribution with variance $\ep^2$. The basic inequality $|x/4|\leq\left|\sin(x/2)\right|\leq|x/2|$, valid for all $x\in[0,\pi]$, 
implies that the $\ep$-scalings of all the moments of $w_{\ep}$ are identical to those of a Gaussian of variance $\ep^2$. In particular, we have that $C_1\ep\leq Z_{\ep}\leq C_2\ep$, for some constants $C_2>C_1>0$. We can thus effectively use the kernel $w_{\ep}$ as if it is a Gaussian of variance $\ep^2$, thus reusing much of scaling considerations (of polynomial type in $\ep^{-1}$ and $N^{-1}$) found in~\cite{Cornalba2018aTR}, where $w_{\ep}$ is Gaussian.

\begin{rem}
Throughout the paper, the quantities in~\eqref{i:311a}--\eqref{i:311b} will always be understood under scalings of the type $N\ep^{\theta}=1$, for $\theta$ large enough. Such a scaling is convenient to deal with the simultaneous limits $\ep\rightarrow 0$ and $N\rightarrow\infty$. This is because most bounds that we will prove with respect to~\eqref{i:311a}--\eqref{i:311b} feature a polynomial contribution in $\ep^{-1}$ and $N^{-1}$, as mentioned above.  
\end{rem}
Section \ref{s:3} is concerned with the evolution of the particle system~\eqref{i:1}. Subsection \ref{ss:3.2} contains Proposition \ref{ip:3}, which provides relative compactness in law for the families $\{\rho_{\ep}\}_{\ep},\{j_{\ep}\}_{\ep}$, and $\{j_{2,\ep}\}_{\ep}$ in the limit $\ep\rightarrow 0$. In this result, the crucial feature of time regularity of the processes is settled not by the Kolmogorov criterion~\cite[Corollary 14.9]{Kallenberg2002a} (as for the corresponding result in~\cite{Cornalba2018aTR}), but by Simon's compactness criterion~\citep[Theorem 5]{Simon1987a} applied in the context of the Prokhorov Theorem~\cite{Kallenberg2002a}.
The need for the latter method arises since the estimates for the time regularity obtained here are less sharp than those in~\cite{Cornalba2018aTR}, due to 
the use of the propagation of chaos (Proposition \ref{ip:1}). 

We then focus on the evolution equations for~\eqref{i:311a}, which are the building blocks of our regularised Dean-Kawasaki model. As the evolution equations for~\eqref{i:311a} are not closable in~\eqref{i:311a}, we rely on three relevant approximations. The first one, explained in Subsection \ref{ss:3.3}, provides the distinctive particle interaction term $\left\{W'\ast \rho_\ep\right\}\rho_{\ep}$. 
The second one, detailed in Subsection \ref{ss:3.4}, gives the relevant Dean--Kawasaki type noise (depending on $\rho_{\ep}$ and on a regular infinite-dimensional noise). The key differences with respect to the analogous argument performed in~\cite{Cornalba2018aTR} (these being primarily due to the use of the propagation of chaos, the use of the von Mises kernels, and the lack of control over inverse powers of $\rho_{\ep}$ in the case of dependent particles) are explained there.   The third and final approximation, which we justify in a low-temperature regime, allows us to replace $j_{2,\ep}$ (defined in~\eqref{i:311b}) with a multiple of $\partial\rho_{\ep}/\partial x$. 

In Section \ref{s:4} we take advantage of the approximations discussed above and derive our \emph{regularised Dean-Kawasaki} model for weakly interacting particles in undamped regime
\begin{subequations}
\label{i:3000}
\begin{empheq}[left={}\empheqlbrace]{align}
  \,\,\displaystyle\frac{\partial \tilde{\rho}_{\ep}}{\partial t}(x,t) & = -\frac{\partial \tilde{j}_{\ep}}{\partial x}(x,t),\label{i:3000a}\\
  \,\,\displaystyle\frac{\partial \tilde{j}_{\ep}}{\partial t}(x,t) & = 
  -\gamma \tilde{j}_{\ep}(x,t)-\left(\frac{\sigma^2}{2\gamma}\right)\frac{\partial \tilde{\rho}_{\ep}}{\partial x}(x,t)-\{W'\ast\tilde{\rho}_{\ep}(\cdot,t)\}(x)\tilde{\rho}_{\ep}(x,t)
  +\frac{\sigma}{\sqrt{N}}\sqrt{\tilde{\rho}_{\ep}(x,t)}\,\tilde{\xi}_{\ep}, \label{i:3000b}\\
  \,\, \displaystyle\tilde{\rho}_{\ep}(x,0) & =\rho_0(x),\quad \tilde{j}_{\ep}(x,0)=j_0(x)\nonumber,
  \end{empheq}
\end{subequations}
for $ (x,t)\in \mathbb{T}\times[0,T]$, where $(\rho_0,j_0$) is a suitable initial datum, where $\tilde{\xi}_{\ep}$ is a regular $Q$-Wiener process (e.g., in the sense of~\cite{Prevot2007a}), and where the aforementioned approximations are visible in the last three terms of the right-hand side of~\eqref{i:3000b}. We use $(\tilde{\rho}_{\ep},\tilde{j}_{\ep})$ to refer to the solution of the approximate model~\eqref{i:3000}, and $({\rho}_{\ep},{j}_{\ep})$ to refer to the original densities in~\eqref{i:311a}.

We provide a few preliminary results concerning the existence of local mild solutions to~\eqref{i:3000} and also to its noise-free version. We then prove the main existence and uniqueness result of the paper, Theorem \ref{ithm:10}. More specifically, we perform a small-noise regime analysis, in a similar way to the one carried out in~\cite{Cornalba2018aTR}, to prove a high-probability existence and uniqueness result of mild solutions to~\eqref{i:3000}. On top of the arguments in~\cite{Cornalba2018aTR}, additional localisation procedures via stopping times and the conservation of mass for the system are needed to treat the locally bounded (superlinear) interaction term $\{W'\ast \tilde{\rho}_\ep\}\tilde{\rho}_{\ep}$.

\section{Preliminary results}\label{s:2}

We prove a few results which will be used in Section \ref{s:3} for the derivation of the undamped regularised Dean--Kawasaki model for weakly interacting particles.

\subsection{Propagation of chaos}\label{ss:2.1}

We first quantify how much the particles in~\eqref{i:1} follow their counterparts in~\eqref{i:2}.

\begin{prop}[Propagation of chaos]\label{ip:1}
Let $N\in\mathbb{N}$, let $\alpha\geq 2$ be an even natural number, let $T>0$, and let $W\in C^2(\mathbb{T})$. There exists a constant $C=C(W,T,\alpha)$ such that
\begin{align}
\sup_{t\in[0,T]}{\mean{|q_1(t)-\ov{q}_1(t)|^{\alpha}+|p_1(t)-\ov{p}_1(t)|^{\alpha}}^{\frac{1}{\alpha}}}\leq \frac{C(W,T,\alpha)}{\sqrt{N}},
\end{align}
where the particle notation is inherited from~\eqref{i:1} and~\eqref{i:2}.
\end{prop}
\begin{proof}
We adapt the proof of~\citep[Theorem 3.3]{Malrieu01}. Let $\beta_N(t):=\mean{|q_1(t)-\ov{q}_1(t)|^{\alpha}+|p_1(t)-\ov{p}_1(t)|^{\alpha}}$. We apply the It\^o formula for the function $f(z)=|z|^\alpha$ applied to the processes $q_i(t)-\ov{q}_i(t)$ and $p_i(t)-\ov{p}_i(t)$, for each $i\in\{1,\dots,N\}$, and sum the results. We notice that the stochastic noise for $p_i(t)-\ov{p}_i(t)$, $i\in\{1,\dots,N\}$, vanishes by assumption. We obtain
\begin{subequations}
\label{i:3}
\begin{empheq}{align}
  \sum_{i=1}^{N}{|q_i(t)-\ov{q}_i(t)|^{\alpha}} & = \int_{0}^{t}{\sum_{i=1}^{N}{\alpha (q_i(r)-\ov{q}_i(r))^{\alpha-1}(p_i(r)-\ov{p}_i(r))\m r}}=:T_1,\label{i:3a}\\
  \sum_{i=1}^{N}{|p_i(t)-\ov{p}_i(t)|^{\alpha}} &= -\frac{\alpha}{N}\int_{0}^{t}{\sum_{i,j=1}^{N}{(p_i(r)-\ov{p}_i(r))^{\alpha-1}}}\left(W'(q_i(r)-q_j(r))-W'\ast\mu_r(\ov{q}_i(r))\right)\m r\nonumber\\
& \quad +\int_{0}^{t}{\sum_{i=1}^{N}{\alpha (p_i(r)-\ov{p}_i(r))^{\alpha-1}(-\gamma\left[p_i(r)-\ov{p}_i(r)\right])\m r}}=:T_2+T_3.\label{i:3b}  
\end{empheq}
\end{subequations}
We bound $T_1$ using the Young inequality with exponents $\alpha$ and $\alpha/(\alpha-1)$. We thus obtain for $T_1+T_3$
\begin{align}\label{i:8}
T_1+T_3\leq C(\alpha,\gamma)\int_{0}^{t}{\sum_{i=1}^{N}{\left(|q_i(r)-\ov{q}_i(r)|^{\alpha}+|p_i(r)-\ov{p}_i(r)|^{\alpha}\right)}\m r}.
\end{align}
As for $T_2$, we rewrite it as $T_2=-\frac{\alpha}{N}\int_{0}^{t}{\sum_{i,j=1}^{N}{\left\{c^{(1)}_{ij}(r)+c^{(2)}_{ij}(r)\right\}\m r}}$, where 
\begin{align*}
c^{(1)}_{ij}(r) & := \left[W'(q_i(r)-q_j(r))-W'(\ov{q}_i(r)-\ov{q}_j(r))\right]\left(p_i(r)-\ov{p}_i(r)\right)^{\alpha-1},\\
c^{(2)}_{ij}(r) & := \left[W'(\ov{q}_i(r)-\ov{q}_j(r))-W'\ast\mu_{r}(\ov{q}_i(r))\right]\left(p_i(r)-\ov{p}_i(r)\right)^{\alpha-1}.
\end{align*}
We use the boundedness of $W''$, a Taylor expansion of $W'$, and the Young inequality with exponents $\alpha$ and $\alpha/(\alpha-1)$ to find
\begin{align}\label{i:7}
& \left|-\frac{\alpha}{N}\int_{0}^{t}{\sum_{i,j=1}^{N}{c^{(1)}_{ij}(r)\m r}}\right| \nonumber\\
& \quad \leq \frac{\alpha}{N}\int_{0}^{t}{\sum_{i,j=1}^{N}{\left|W'(q_i(r)-q_j(r))-W'(\ov{q}_i(r)-\ov{q}_j(r))\right|\left|p_i(r)-\ov{p}_i(r)\right|^{\alpha-1}}\m r}\nonumber\\
& \quad \leq \frac{C(W,\alpha)}{N}\int_{0}^{t}{\sum_{i,j=1}^{N}{\left\{\left|q_i(r)-\ov{q}_i(r)\right|+\left|q_j(r)-\ov{q}_j(r)\right|\right\}\left|p_i(r)-\ov{p}_i(r)\right|^{\alpha-1}}\m r}\nonumber\\
& \quad\leq \frac{C(W,\alpha)}{N}\int_{0}^{t}{\sum_{i,j=1}^{N}{\left\{\left|q_i(r)-\ov{q}_i(r)\right|^{\alpha}+\left|q_j(r)-\ov{q}_j(r)\right|^{\alpha}+\left|p_i(r)-\ov{p}_i(r)\right|^{\alpha}\right\}}\m r}\nonumber\\
&\quad = C(W,\alpha)\int_{0}^{t}{\sum_{i=1}^{N}{\left\{\left|q_i(r)-\ov{q}_i(r)\right|^{\alpha}+\left|p_i(r)-\ov{p}_i(r)\right|^{\alpha}\right\}}\m r}.
\end{align} 
Fix $r\in[0,t]$ and $i\in\{1,\dots,N\}$. We employ the H\"older inequality with exponents $\alpha$ and $\alpha/(\alpha-1)$ to obtain
\begin{align}\label{i:9}
\mean{\sum_{j=1}^{N}c^{(2)}_{ij}(r)} & =\mean{\sum_{j=1}^{N}{\left[W'(\ov{q}_i(r)-\ov{q}_j(r))-W'\ast\mu_{r}(\ov{q}_i(r))\right]\left(p_i(r)-\ov{p}_i(r)\right)^{\alpha-1}}}\nonumber\\
& \leq \mean{\left|p_i(r)-\ov{p}_i(r)\right|^{\alpha}}^{(\alpha-1)/\alpha}\theta_i^{1/\alpha}(r),
\end{align}
where 
\begin{align*}
\theta_i(r):=\mean{\left|\sum_{j=1}^{N}{\xi_{\ov{q}_i(r),\ov{q}_j(r)}}\right|^{\alpha}}=\mean{\left(\sum_{j=1}^{N}{\xi_{\ov{q}_i(r),\ov{q}_j(r)}}\right)^{\alpha}},
\end{align*}
with $\xi_{\ov{q}_i(r),\ov{q}_j(r)}:=W'(\ov{q}_i(r)-\ov{q}_j(r))-W'\ast\mu_{r}(\ov{q}_i(r))$, 
and where we have also used the fact that $\alpha$ is an even natural number. We define
\begin{align*}
\mathcal{T}_{1,\alpha} & :=\left\{\mathbf{j}=(j_1,\dots,j_{\alpha})\in\{1,\dots,N\}^\alpha:\exists j_k\neq i\mbox{ such that }j_k\mbox{ appears exactly once in }\mathbf{j}\right\},\\
\mathcal{T}_{2,\alpha} & :=\left\{\mathbf{j}=(j_1,\dots,j_{\alpha})\in\{1,\dots,N\}^\alpha:\mathbf{j}\notin\mathcal{T}_{1,\alpha}\right\}.
\end{align*} 
We have $\#\mathcal{T}_{2,\alpha}\leq C(\alpha)N^{\alpha/2}$, where $\#$ denotes set cardinality. To see this, consider a generic $\mathbf{j}\in\mathcal{T}_{2,\alpha}$. There are at most $\alpha/2$ values attained in $\mathbf{j}$: arguing by contradiction, if this is not the case, then $i$ is attained exactly once (due to the definition of $\mathcal{T}_{2,\alpha}$). However, this means that the remaining $\alpha-1$ occurrences of $\mathbf{j}$ are distributed among at least $\alpha/2$ values, granting the existence of $j_k\neq i$ appearing exactly once in $\mathbf{j}$, and thus contradicting the definition of $\mathcal{T}_{2,\alpha}$.
We therefore have no more than $C(\alpha)N^{\alpha/2}$ possible configurations in $\mathcal{T}_{2,\alpha}$, where $C(\alpha)$ is a suitable constant.
We expand 
the definition of $\theta_i(r)$ as 
\begin{align*}
\theta_i(r)=\sum_{\mathbf{j}\in\mathcal{T}_{1,\alpha}}{\mean{\prod_{k=1}^{\alpha}{\xi_{\ov{q}_i(r),\ov{q}_{j_k}(r)}}}}+\sum_{\mathbf{j}\in\mathcal{T}_{2,\alpha}}{\mean{\prod_{k=1}^{\alpha}{\xi_{\ov{q}_i(r),\ov{q}_{j_k}(r)}}}}.
\end{align*}
For any $\mathbf{j}\in\mathcal{S}_{1,\alpha}$, it holds that $\mean{\prod_{k=1}^{\alpha}{\xi_{\ov{q}_i(r),\ov{q}_{j_k}(r)}}}=0$. To see this, let $z\in\mathbb{T}$, and let $j\neq i$ be an index appearing just once in $\mathbf{j}$.
Then
\begin{align}
& \mean{\left.\prod_{k=1}^{\alpha}{\xi_{\ov{q}_i(r),\ov{q}_{j_k}(r)}}\right|\ov{q}_i(r)=z} 
 = \prod_{j_k=i}{\xi_{z,z}}\cdot\mean{\left.\left(\prod_{j_k\neq i,j_k\neq j}{\xi_{z,\ov{q}_{j_k}(r)}}\right)\xi_{z,\ov{q}_{j}(r)}\right|\ov{q}_i(r)=z}\nonumber\\
& \quad = \prod_{j_k=i}{\xi_{z,z}}\cdot\mean{\left(\prod_{j_k\neq i,j_k\neq j}{\xi_{z,\ov{q}_{j_k}(r)}}\right)\xi_{z,\ov{q}_{j}(r)}}\label{eq:1000}\\
& \quad = \prod_{j_k=i}{\xi_{z,z}}\cdot\mean{\prod_{j_k\neq i,j_k\neq j}{\xi_{z,\ov{q}_{j_k}(r)}}}\mean{\xi_{z,\ov{q}_{j}(r)}}\label{eq:1001}\\
& \quad = \prod_{j_k=i}{\xi_{z,z}}\cdot\mean{\prod_{j_k\neq i,j_k\neq j}{\xi_{z,\ov{q}_{j_k}(r)}}}\mean{\left(W'(z-\ov{q}_{j}(r))-W'\ast\mu_{r}(z)\right)} = 0,\label{eq:1002}
\end{align}
where independence of particles is used in~\eqref{eq:1000} and~\eqref{eq:1001}, and $\mean{\left(W'(z-\ov{q}_{j}(r))-W'\ast\mu_{r}(z)\right)}=0$ settles~\eqref{eq:1002}. The exchangeability of particles, the H\"older inequality, the boundedness of $W'$, and the bound $\#\mathcal{T}_{2,\alpha}\leq C(\alpha)N^{\alpha/2}$ then give
\begin{align}\label{i:12}
\theta_i(r) & = \sum_{\mathbf{j}\in\mathcal{S}_{2,\alpha}}{\mean{\prod_{k=1}^{\alpha}{\xi_{\ov{q}_i(r),\ov{q}_{j_k}(r)}}}} \nonumber\\
& \leq C(\alpha)N^{\frac{\alpha}{2}}\mean{\left|W'(\ov{q}_1(r)-\ov{q}_2(r))\right|^{\alpha}+\left|W'\ast\mu_r(\ov{q}_1(r))\right|^{\alpha}}\leq C(W,\alpha)N^{\frac{\alpha}{2}}.
\end{align}
We sum~\eqref{i:3a} and~\eqref{i:3b}, combine~\eqref{i:8},~\eqref{i:7},~\eqref{i:9}, and~\eqref{i:12}, and use the exchangeability of the particles to obtain
\begin{align}\label{i:10}
\beta_N(t)\leq \int_{0}^{t}{C(\alpha,\gamma)\beta_{N}(r)\m r}+\int_{0}^{t}{C(W,\alpha)N^{-1/2}(\beta_N(r))^{(\alpha-1)/\alpha}\m r}.
\end{align}
Applying the Young inequality in the second integral of~\eqref{i:10} and then Gronwall's inequality completes the proof.
\end{proof}
We point out a couple of differences between Proposition \ref{ip:1} and~\citep[Theorem 3.3]{Malrieu01}. Firstly, we do not require convexity for the interaction potential $W$, as we are only interested in an estimate up to a given finite time; there is thus no need for a dissipative term in~\eqref{i:10}. 
Secondly, since the derivative $W'$ is bounded, we can choose $\alpha$ arbitrarily large without violating the validity of~\eqref{i:12}. In the proof of Proposition \ref{ip:4} below, we will pick $\alpha>2$.

\subsection{Fokker--Planck regularity estimates}\label{ss:2.2}

We now establish useful regularity properties of the particle system~\eqref{i:2}. We use $C^n$ to denote $n$ times continuously differentiable functions on $\mathbb{T}$, for $n\in\mathbb{N}\cup\{0\}$. We first specify our assumptions on~\eqref{i:2}.

\begin{ass}\label{i:7000}
  We assume that the initial datum $(\ov{q}(0),\ov{p}(0))$ of~\eqref{i:2} coincides with $(\ov{q}_{\aux}(t_0),\ov{p}_{\aux}(t_0))$
  for some $t_0>0$, where $(\ov{q}_{\aux},\ov{p}_{\aux})$ is an auxiliary process satisfying $\eqref{i:2}$ and starting from an
  initial datum distributed according to a probability density $f_0$ satisfying
  \begin{align*}
    \int_{\mathbb{T}}{\int_{\mathbb{R}}{f_0(q,p)(1+p^2)^k\emph{\m} p}\emph{\m} q}<\infty.
  \end{align*} 
\end{ass}

Our choice to only consider a process ``restarted'' at some time $t_0>0$ is motivated by the need of the uniform-in-time Sobolev estimates found in~\citep[(17.2)]{Villani2009b}, which we will use in the following result.
\begin{prop}\label{ip:2}
For $n,n_1\in\mathbb{N}\cup\{0\}$ and $c\geq 2$, let $w$ be a $C^n$-probability density function and $g\in C^n$. Let the initial datum of~\eqref{i:2} be as specified in Assumption \ref{i:7000}.
Then
\begin{align*}
\int_{\mathbb{T}}{\left|\mean{{\color{blue1}g(\ov{q}(t))}\ov{p}^{n_1}(t)\frac{\partial^n}{\partial x^n}w(x-\ov{q}(t))}\right|^c\emph{\m} x}\leq C(g,t_0,f_0,n),\qquad\mbox{for all }t\geq 0,
\end{align*}
where $C(g,t_0,f_0,n)$ does not depend on $w$.
\end{prop}
\begin{proof}
We first prove that, for $f_t(q,p)$ being the probability density function of $(\ov{q}(t),\ov{p}(t))$ and for any $\tilde{g}\in C^0$, we have
\begin{align}\label{i:16}
\int_{\mathbb{T}}{\left|\int_{\mathbb{R}}{\left|{\color{blue1}\tilde{g}(q)}p^{n_1}\right|\left|\frac{\partial^m}{\partial q^m}f_t(q,p)\right|\m p}\right|^c\m q}\leq C(\tilde{g},t_0,f_0,n),\qquad \mbox{for }m\in\{0,1,\dots,n\}.
\end{align}
We use the boundedness of $g$ and the H\"older inequality with exponents $c$ and $c/(c-1)$ to obtain
\begin{align}\label{i:15}
& \int_{\mathbb{T}}{\left|\int_{\mathbb{R}}{|{\color{blue1}\tilde{g}(q)}p^{n_1}|\left|\frac{\partial^m}{\partial q^m}f_t(q,p)\right|\m p}\right|^c\m q} 
\leq C(\tilde{g})\int_{\mathbb{T}}{\left|\int_{\mathbb{R}}{|p^{n_1}|\left|\frac{\partial^m}{\partial q^m}f_t(q,p)\right|^{\frac{2}{c}}\left|\frac{\partial^m}{\partial q^m}f_t(q,p)\right|^{\frac{c-2}{c}}\m p}\right|^c\m q}\nonumber\\
& \quad \leq C(\tilde{g})\int_{\mathbb{T}}{\left(\int_{\mathbb{R}}{|p^{n_1c}|\left|\frac{\partial^m}{\partial q^m}f_t(q,p)\right|^{2}\left|1+p^2\right|^{kc}\m p}\right)\left(\int_{\mathbb{R}}{\left|\frac{\partial^m}{\partial q^m}f_t(q,p)\right|^{\frac{c-2}{c-1}}\left|1+p^2\right|^{-\frac{kc}{c-1}}\m p}\right)^{c-1}\m q}.
\end{align}
The second $p$-integral in~\eqref{i:15} can be bounded by a constant $C(t_0,f_0,n)$, provided we pick $k>\frac{c-1}{2c}$. To see this, we notice that~\citep[(17.2)]{Villani2009b} gives uniform bounds in time for $\|f_t\|_{W^{n+2,2}(\mathbb{T}\times\mathbb{R})}$, where we use the Sobolev space notation. The continuous embedding $W^{n+2,2}(\mathbb{T}\times\mathbb{R})\subset C^{m}(\mathbb{T}\times\mathbb{R})$, which is a result of the application of~\citep[Theorem 4.12, Part I, Case A, equation (1)]{Adams2003a}) thus implies that
\begin{align*}
\sup_{q\in\mathbb{T},p\in\mathbb{R}}{\left|\frac{\partial^m}{\partial q^m}f_t(q,p)\right|}\leq C(t_0,f_0,n),\qquad \mbox{for all } t\geq 0.
\end{align*}
As a result, the argument of the second $p$-integral in~\eqref{i:15} is controlled by $(1+p^2)^{-\frac{kc}{c-1}}$, which is integrable thanks to the choice of $k$. Thus~\eqref{i:15} is bounded by 
\begin{align*}
C(\tilde{g},t_0,f_0,n)\int_{\mathbb{T}}{\int_{\mathbb{R}}{|p^{n_1c}|\left|\frac{\partial^m}{\partial q^m}f_t(q,p)\right|^{2}\left|1+p^2\right|^{kc}\m p}\m q},
\end{align*}
which is in turn uniformly bounded in time due to~\citep[(17.2)]{Villani2009b}. We have thus verified~\eqref{i:16}.
We now define $\tilde{f}_t(q):=\int_{\mathbb{R}}{(\partial^n/\partial q^n)\left\{{\color{blue1}g(q)}p^{n_1}f_t(q,p)\right\}\m p}$. We use integration by parts and Young's inequality for convolutions to bound
\begin{align}\label{i:17}
&\int_{\mathbb{T}}{\left|\mean{{\color{blue1}g(\ov{q}(t))}\ov{p}^{n_1}(t)\frac{\partial^n}{\partial x^n}w(x-\ov{q}(t))}\right|^c\m x} = \int_{\mathbb{T}}{\left|\int_{\mathbb{T}}{\int_{\mathbb{R}}{{\color{blue1}g(q)}p^{n_1}f_t(q,p)\frac{\partial^n}{\partial q^n}w(x-q)\m p}\m q}\right|^c\m x} \nonumber\\
& \quad = \int_{\mathbb{T}}{\left|\int_{\mathbb{T}}{\int_{\mathbb{R}}{w(x-q)\frac{\partial^n}{\partial q^n}\left\{{\color{blue1}g(q)}p^{n_1}f_t(q,p)\right\}\m p}\m q}\right|^c\m x}\nonumber\\
& \quad = \int_{\mathbb{T}}{\left|\int_{\mathbb{T}}{w(x-q) \tilde{f}_t(q)\m q}\right|^c\m x}=\left\|w\ast\tilde{f}_t\right\|^c_{L^c(\mathbb{T})}\leq \left\| w\right\|^c_{L^1(\mathbb{T})}\left\|\tilde{f}_t\right\|^c_{L^c(\mathbb{T})} = \left\|\tilde{f}_t\right\|^c_{L^c(\mathbb{T})}\nonumber\\
&\quad = \int_{\mathbb{T}}{\left|\int_{\mathbb{R}}{\frac{\partial^n}{\partial q^n}\left\{{\color{blue1}g(q)}p^{n_1}f_t(q,p)\right\}\m p}\right|^c\m q}\nonumber\\
&\quad \leq C(n,c)\sum_{j=0}^{n}{\int_{\mathbb{T}}{\left|\int_{\mathbb{R}}{\left|\frac{\partial^j}{\partial q^j}\left\{{\color{blue1}g(q)}\right\}p^{n_1}\frac{\partial^{n-j}}{\partial q^{n-j}}\left\{f_t(q,p)\right\}\right|\m p}\right|^c\m q}}.
\end{align}
As $g\in C^n$, it is clear that each of the $(n+1)$ terms in~\eqref{i:17} is as prescribed by the left-hand-side of~\eqref{i:16}, for some appropriate choices of $\tilde{g}$ and $m$. The proof is complete.
\end{proof}
\begin{rem}
The use of ~\citep[(17.2)]{Villani2009b} is the reason for having $\mathbb{T}$, and not $\mathbb{R}$, as the spatial domain. 
\end{rem}
\begin{rem}\label{irem:1}
With the same notation and assumptions of Propositions \ref{ip:1} and \ref{ip:2}, let the initial datum of the particles systems~\eqref{i:1} and~\eqref{i:2} have density $(\ov{q}_{\aux}(t_0),\ov{p}_{\aux}(t_0))$. It is easy to prove that the particle systems~\eqref{i:1} and~\eqref{i:2} have moments of any order uniformly bounded on $[0,T]$. This is a simple consequence of the boundedness of $W'$.\end{rem}

\subsection{A useful application of the propagation of chaos}\label{ss:2.3}

The result proved in this subsection is used in Section \ref{s:3} in order to provide estimates independent of $\ep$ for the $H^k$-norm of the expressions~\eqref{i:311a} and~\eqref{i:311b}. We use the standard Sobolev space notation $H^k:=H^k(\mathbb{T})$, for $k\in\mathbb{N}$, and also $L^p:=L^p(\mathbb{T})$, for $p\in[1,\infty]$. As already mentioned, we will always assume a scaling of type $N\ep^{\theta}=1$, for $\theta$ large enough, say $\theta>\theta_0$. In this paper, we are not interested in optimising in $\theta$ (i.e., in finding its lowest admissible value). 

\begin{prop}\label{ip:4}
Let the assumptions of Propositions \ref{ip:1} and \ref{ip:2} be satisfied, and let $\mathbb{N}\ni c\geq 2$. Then, in the regime $N\ep^{\theta}=1$, for $\theta$ large enough, we have that
\begin{subequations}\label{i:25}
\begin{align}
& \mean{\left\|\frac{1}{N}\sum_{i=1}^{N}{p^{n_1}_i(t)\frac{\partial^n}{\partial^n x}w_{\ep}(\cdot-q_i(t))}\right\|^c_{L^c}} \label{i:25a}\intertext{and}
& \mean{\left\|\frac{1}{N}\sum_{i=1}^{N}{\left\{\frac{1}{N}\sum_{j=1}^{N}{W'(q_i(t)-q_j(t))}\right\}p^{n_1}_i(t)\frac{\partial^n}{\partial^n x}w_{\ep}(\cdot-q_i(t))}\right\|^c_{L^c}}\label{i:25b}
\end{align}
\end{subequations}
are uniformly bounded in $\ep$, $N$, and $t\in[0,T]$.
\end{prop}
Even though the proof of Proposition \ref{ip:4} is a suitable extension of~\cite[Proof of Proposition 1.1]{Cornalba2018aTR}, we include it here to keep the paper as self-contained as possible. For the benefit of the curious reader, we point out the analogies between the two proofs in the subsequent Remark \ref{irem:40}, which may be skipped on a first reading. \begin{proof}[Proof of Proposition \ref{ip:4}]
We first deal with~\eqref{i:25a}. Set $a_i(x,t):=p^{n_1}_i(t)\frac{\partial^n}{\partial^n x}w_{\ep}(x-q_i(t))$. If we expand the $L^c$-norm, we get
\begin{align*}
& \mean{\left\|\frac{1}{N}\sum_{i=1}^{N}{p^{n_1}_i(t)\frac{\partial^n}{\partial^n x}w_{\ep}(\cdot-q_i(t))}\right\|^c_{L^c}}\\
& \quad = \frac{1}{N^c}\sum_{\mathbf{j}\in \mathcal{S}_{1,c}}{\mean{\int_{\mathbb{T}}{\prod_{k=1}^{c}{a_{j_k}(x,t)}}\m x}}+\frac{1}{N^c}\sum_{\mathbf{j}\in \mathcal{S}_{2,c}}{\mean{\int_{\mathbb{T}}{\prod_{k=1}^{c}{a_{j_k}(x,t)}}\m x}},
\end{align*}
where $\mathcal{S}_{1,c}$ and  $\mathcal{S}_{2,c}$ are given by
\begin{subequations}\label{i:24}
\begin{align}
\mathcal{S}_{1,c} & := \left\{\mathbf{j}=(j_1,\dots,j_{c})\in\{1,\dots,N\}^c:\mathbf{j}\mbox{ does not have repeated components}\right\},\label{i:24a}\\
\mathcal{S}_{2,c} & := \left\{\mathbf{j}=(j_1,\dots,j_{c})\in\{1,\dots,N\}^c:\mathbf{j}\mbox{ has repeated components}\right\}.\label{i:24b}
\end{align}
\end{subequations}
We use the exchangeability of the particles, the fact that $\# \mathcal{S}_{2,c}\leq C(c)N^{c-1}$, the H\"older inequality, and the fact that all moments of $p_i$ are uniformly bounded on $[0,T]$ (see Remark \ref{irem:1}) to obtain 
\begin{align}\label{i:2000}
\frac{1}{N^c}\sum_{\mathbf{j}\in \mathcal{S}_{2,c}}{\mean{\int_{\mathbb{T}}{\prod_{k=1}^{c}{a_{j_k}(x,t)}\m x}}}\leq \frac{C(c)}{N}\int_{\mathbb{T}}{\mean{a^c_1(x,t)}\m x}\leq \frac{\mathbb{Q}(\ep^{-1})}{N}\rightarrow 0\qquad\mbox{as }\ep\rightarrow 0,
\end{align}
where $\mathbb{Q}$ is some polynomial whose degree depends on $n$. The convergence to zero is granted by the scaling $N\ep^{\theta}=1$, assuming that $\theta$ is large enough. 
For each $\mathbf{j}\in\mathcal{S}_{1,c}$, we now analyse $\mean{\int_{\mathbb{T}}{\prod_{k=1}^{c}{a_{j_k}(x,t)}\m x}}$. The particles $\{(q_i,p_i)\}_{i=1}^{N}$ not being independent, we rely on the propagation of chaos, i.e., on Proposition \ref{ip:1}. The strategy is the following: in each $a_{j_k}(x,t)$, we add and subtract relevant quantities associated with~\eqref{i:2}. More specifically, we split
\begin{subequations}\label{i:26}
\begin{align}
p^{n_1}_i(t) & = \underbrace{p^{n_1}_i(t)-\ov{p}^{n_1}_i(t)}_{A_{1,i}:=} + \underbrace{\ov{p}^{n_1}_i(t)}_{B_{1,i}:=},\label{i:26a}\\
\frac{\partial^n}{\partial^n x}w_{\ep}(x-q_i(t)) & = \underbrace{\frac{\partial^n}{\partial^n x}w_{\ep}(x-q_i(t))-\frac{\partial^n}{\partial^n x}w_{\ep}(x-\ov{q}_i(t))}_{A_{2,i}:=} + \underbrace{\frac{\partial^n}{\partial^n x}w_{\ep}(x-\ov{q}_i(t))}_{B_{2,i}:=}\label{i:26b}.
\end{align}
\end{subequations}
The estimates
\begin{subequations}\label{i:200}
\begin{align}
|A_{1,i}| & \leq C(n_1)|p_i(t)-\ov{p}_i(t)|(|p_i(t)|^{n_1-1}+|\ov{p}_i(t)|^{n_1-1}),\label{i:200a}\\
|A_{2,i}| & \leq \mathbb{Q}({\ep}^{-1})|q_i(t)-\ov{q}_i(t)|,\label{i:200b}\\
|B_{2,i}| & \leq \mathbb{Q}({\ep}^{-1})\label{i:200c},
\end{align}
\end{subequations}
where $\mathbb{Q}$ is a polynomial, follow easily from Taylor expansions and bounds on derivatives of $w_{\ep}$.
We regroup the $2^{2c}$ terms arising from the expansion of the product $\prod_{k=1}^{c}{\left(A_{1,j_k}+B_{1,j_k}\right)\left(A_{2,j_k}+B_{2,j_k}\right)}$ as
\begin{align*}
\prod_{k=1}^{c}{\left(A_{1,j_k}+B_{1,j_k}\right)\left(A_{2,j_k}+B_{2,j_k}\right)}=\prod_{k=1}^{c}{B_{1,j_k}B_{2,j_k}}+\sum_{s=1}^{2^{2c}-1}{C_s},
\end{align*}
where the sum spans all $2^{2c}-1$ terms of the expansion which feature at least one factor of type $A$ (i.e., each $C_s$ is a product of $2c$ terms of type $A$ and $B$, with at least one being of type $A$). As a result, we write
\begin{align}\label{i:2001}
\mean{\int_{\mathbb{T}}{\prod_{k=1}^{c}{a_{j_k}(x,t)}\m x}} = \mean{\int_{\mathbb{T}}{\prod_{k=1}^{c}{B_{1,j_k}B_{2,j_k}\m x}}}+\sum_{s=1}^{2^{2c}-1}{\mean{\int_{\mathbb{T}}{C_s\m x}}}:=T_1+T_2.
\end{align}
We bound $T_2$. As each term $C_s$ contains a factor of type $A$, we can use~\eqref{i:200} to deduce that 
\begin{align*}
|C_s| & \leq \left(\prod_{i=1}^{c}{|p_i(t)-\ov{p}_i(t)|^{\alpha_i}|q_i(t)-\ov{q}_i(t)|^{\beta_i}}\right)\\
& \quad\times\left(\prod_{i=1}^{c}{\left[C(n_1)(|p_i(t)|^{n_1-1}+|\ov{p}_i(t)|^{n_1-1})\right]^{\alpha_i}\left[\mathbb{Q}({\ep}^{-1})\right]^{\beta_i}}\right)\\
& \quad\times\left(\prod_{i=1}^{c}{\left[\ov{p}_i(t)\right]^{1-\alpha_i}\left[\mathbb{Q}({\ep}^{-1})\right]^{1-\beta_i}}\right)=:T_3\times T_4\times T_5,
\end{align*}
for some $\alpha_i,\beta_i\in\{0;1\}$, $\sum_{i=1}^{c}{\alpha_i+\beta_i}\in\{1,\dots,2c\}$. We can bound $\mean{|C_s|}$ by applying a multi-factor H\"older inequality involving each term of the product $\mean{T_3\times T_4\times T_5}$. More precisely, the expectation of each term of $T_3$ is either unitary, or dealt with by using Proposition \ref{ip:1} (propagation of chaos); the expectation of each term of $T_4$ and $T_5$ is either unitary, or dealt with by relying on the fact that all moments of $\ov{p}_i(t)$, $p_i(t)$ are uniformly bounded on $[0,T]$, see Remark \ref{irem:1}. Due to the constraint $\sum_{i=1}^{c}{\alpha_i+\beta_i}\in\{1,\dots,2c\}$, we can apply Proposition \ref{ip:1} at least once. Thus $\mean{|C_s|}\leq C(n_1)N^{-\gamma_1}\ep^{-\gamma_2}$, for some $\gamma_1,\gamma_2>0$, for $s=1,\dots,2^{2c}-1$. Provided that $\theta$ is large enough, we deduce that $T_2\rightarrow 0$ as $\ep\rightarrow 0$.

As for $T_1$,  
we rely on independence and identical distribution of the particles $\{(\ov{q}_i,\ov{p}_i)\}_{i=1}^{N}$ and write
\begin{align*}
\mean{T_1} & = \mean{\int_{\mathbb{T}}{\prod_{k=1}^{c}{\ov{p}^{n_1}_i(t)\frac{\partial^n}{\partial^n x}w_{\ep}(x-\ov{q}_i(t))}\m x}}\\
& \leq \int_{\mathbb{T}}{\left|\mean{\ov{p}^{n_1}_1(t)\frac{\partial^n}{\partial x^n}w(x-\ov{q}_1(t))}\right|^c\m x}\leq C(t_0,f_0,n),
\end{align*}
where the last inequality is given by Proposition \ref{ip:2}. The expectation in~\eqref{i:25a} is thus dealt with.

As for the expectation in~\eqref{i:25b}, the analysis proceeds similarly, and we only sketch the relevant details. We may think of the argument of the $L^c$-norm as a sum over two indexes $i,j=1,\dots,N$, thus defining $a_{i,j}(x,t):={W'(q_i(t)-q_j(t))p^{n_1}_i(t)\frac{\partial^n}{\partial^n x}w_{\ep}(x-q_i(t))}$. We split the $L^c$-norm expansion into the contributions given over the index sets $\mathcal{S}_{1,2c}$ and $\mathcal{S}_{2,2c}$ ($c$ couples of indexes). The expectation associated with the index set $\mathcal{S}_{2,2c}$ vanishes in the limit $\ep\rightarrow 0$, using the same arguments leading to~\eqref{i:2000}. Now 
fix $\mathbf{j}\in\mathcal\mathcal{S}_{1,2c}$. If we add the rewriting 
\begin{align*}
W'(q_i(t)-q_j(t))=\underbrace{W'(q_i(t)-q_j(t))-W'(\ov{q}_i(t)-\ov{q}_j(t))}_{A_{3,i,j}:=}+\underbrace{W'(\ov{q}_i(t)-\ov{q}_j(t))}_{B_{3,i,j}:=}
\end{align*} 
to those in~\eqref{i:26}, with the associated bound 
\begin{align*}
|A_{3,i,j}|\leq C(W)\left\{|q_i(t)-\ov{q}_i(t)|+|q_j(t)-\ov{q}_j(t)|\right\}
\end{align*}
we may then write 
\begin{multline}
\mean{\int_{\mathbb{T}}{\prod_{k=1}^{c}{a_{j_{2k-1},j_{2k}}(x,t)}\m x}}  = 
\mean{\int_{\mathbb{T}}{\prod_{k=1}^{c}{B_{1,j_{2k-1}}B_{2,j_{2k-1}}B_{3,j_{2k-1},j_{2k}}\m x}}}+\sum_{s=1}^{2^{3c}-1}{\mean{\int_{\mathbb{T}}{C_s\m x}}}\\
 =: T_1+T_2,
\end{multline}
where the notation is in analogy to~\eqref{i:2001}. The convergence $T_2\rightarrow 0$ is settled as in the first part of the proof, and we omit the details. To bound $T_1$, we simply need to bound 
\begin{align}\label{i:27}
\int_{\mathbb{T}}{\left|\mean{W'(\ov{q}_1(t)-\ov{q}_2(t))\ov{p}^{n_1}_1(t)\frac{\partial^n}{\partial^n x}w_{\ep}(x-\ov{q}_1(t))}\right|^{c}\m x},
\end{align}
where we have used again independence and identical distribution of the particles $\{(\ov{q}_i,\ov{p}_i)\}_{i=1}^{N}$. We notice that
\begin{align*}
& \mean{W'(\ov{q}_1(t)-\ov{q}_2(t))\ov{p}^{n_1}_1(t)\frac{\partial^n}{\partial^n x}w_{\ep}(x-\ov{q}_1(t))\left|\ov{q}_1(t)=z_q,\ov{p}_1(t)=z_p\right.}\\
& \quad = z^{n_1}_p\frac{\partial^n}{\partial^n x}w_{\ep}(x-z_q)\mean{W'(z_q-\ov{q}_2(t))\left|\ov{q}_1(t)=z_q,\ov{p}_1(t)=z_p\right.}\\
& \quad = z^{n_1}_p\frac{\partial^n}{\partial^n x}w_{\ep}(x-z_q)\mean{W'(z_q-\ov{q}_2(t))} =  z^{n_1}_p\frac{\partial^n}{\partial^n x}w_{\ep}(x-z_q)W'\ast \mu_t(z_q),
\end{align*} 
which implies
\begin{align*}
\mean{W'(\ov{q}_1(t)-\ov{q}_2(t))\ov{p}^{n_1}_1(t)\frac{\partial^n}{\partial^n x}w_{\ep}(x-\ov{q}_1(t))}=\mean{W'\ast \mu_t(\ov{q}_1(t))\ov{p}^{n_1}_1(t)\frac{\partial^n}{\partial^n x}w_{\ep}(x-\ov{q}_1(t))}.
\end{align*}
The above equality shows that~\eqref{i:27} is of the form prescribed by Proposition \ref{ip:2}, for $g:=W'\ast \mu_t$; as a matter of fact, $W'\ast \mu_t\in C^n$ because of the uniform regularity of $\mu_{t}$ for $t\in[0,T]$, see~\cite[(17.2)]{Villani2009b}. This ends the proof.
\end{proof}
\begin{rem}\label{irem:40}
The proof of Proposition \ref{ip:4} is built on two splittings. The first one separates the index set in $\mathcal{S}_{1,c}$, $ \mathcal{S}_{2,c}$ (and also $\mathcal{S}_{1,2c}$, $\mathcal{S}_{2,2c}$); the second one distinguishes terms of type $A$ and $B$ for every element in $\mathcal{S}_{1,c}$ (and also in $\mathcal{S}_{1,2c}$). The first splitting benefits from scaling arguments (in $N,\ep$) which are found also in~\cite[Proposition 1.1]{Cornalba2018aTR} (see the distinction between terms $\mathsf{ct}$, and $I_1$--$I_4$ therein). The second splitting benefits from Propagation of chaos, and does not have a counterpart in~\cite[Proposition 1.1]{Cornalba2018aTR}.
\end{rem}
\begin{rem}
In the proof of Proposition \ref{ip:4}, the minimum power $\alpha$ that we need to employ when using the  propagation of chaos is $\alpha=2c$ (for~\eqref{i:25a}) and power $\alpha=3c$ (for~\eqref{i:25b}). In the case of~\eqref{i:25a}, this can be seen easily from the multi-factor H\"older inequality used to deal with the one term $\mean{|C_s|}$ for which $\sum_{i=1}^{c}{\alpha_i+\beta_i}=2c$. An analogous consideration holds for~\eqref{i:25b}. This justifies the need for the propagation of chaos for $\alpha>2$.
\end{rem}

\section{Evolution of the weakly interacting particle system}\label{s:3}

We analyse the time evolution of the densities~\eqref{i:311a}--\eqref{i:311b} and start by deriving the relevant evolution equations.
\begin{lemma} The evolution equations for $\rho_{\ep},j_{\ep}$, and $j_{2,\ep}$ are given by
\begin{subequations}
\begin{align}
\frac{\partial \rho_{\ep}}{\partial t}(x,t) & = -\frac{\partial j_{\ep}}{\partial x}(x,t),\label{i:20a}\\
\frac{\partial j_{\ep}}{\partial t}(x,t) & = -\gamma 
      j_{\ep}(x,t)-j_{2,\ep}(x,t)-\frac{1}{N}\sum_{i=1}^{N}{\left(\frac{1}{N}\sum_{j=1}^{N}{W'(q_i(t)-q_j(t))}\right)w_{\ep}(x-q_i(t))}\nonumber\\
      & \quad +\underbrace{\frac{\sigma}{N}\sum_{i=1}^{N}{w_{\ep}(x-q_i(t))\dot \beta_i}}_{=:\dot{\mathcal{Z}}_N(x,t)}\label{i:20b},\\
\frac{\partial j_{2,\ep}}{\partial t}(x,t) & = -2\gamma 
      j_{2,\ep}(x,t)-j_{3,\ep}(x,t)-\frac{2}{N}\sum_{i=1}^{N}{\left(\frac{1}{N}\sum_{j=1}^{N}{W'(q_i(t)-q_j(t))}\right)p_i(t)w'_{\ep}(x-q_i(t))}\nonumber\\
      & \quad + \sigma^2\frac{\partial \rho_{\ep}}{\partial x}(x,t) + \frac{\sigma}{N}\sum_{i=1}^{N}{2p_{i}(t)w'_{\ep}(x-q_i(t))\dot \beta_i}\label{i:20c},
\end{align}
\end{subequations}
where $j_{3,\ep}:=N^{-1}\sum_{i=1}^{N}{p_i^3(t)w''_{\ep}(x-q_i(t))}$.
\end{lemma}
The proof of the lemma above is a simple application of the It\^o formula, and thus omitted. 

\subsection{Compactness argument}\label{ss:3.2}

We now turn to the main result of this section. 

\begin{prop}\label{ip:3}
Let $T>0$. Let the assumptions of Propositions \ref{ip:1} and \ref{ip:2} be satisfied. Assume the scaling $N\ep^{\theta}=1$, for $\theta$ large enough. The families of processes $\{\rho_{\ep}\}_{\ep}$, $\{j_{\ep}\}_{\ep}$, and $\{j_{2,\ep}\}_{\ep}$ are tight (hence relatively compact in distribution) in $C(0,T;L^2)$, as $\ep\rightarrow 0$. 
\end{prop}
\begin{proof}
Assume for the time being (we will show this below) that 
\begin{align}\label{i:28}
\mean{\|\rho_{\ep}\|_{\mathscr{U}}}, \quad\mean{\|j_{\ep}\|_{\mathscr{U}}},\quad \mean{\|j_{2,\ep}\|_{\mathscr{U}}}\quad \mbox{ are uniformly bounded as }\ep\rightarrow 0,
\end{align}
where $\|\cdot\|_{\mathscr{U}}$ is the natural norm of the space 
\begin{align}\label{i:4000}
\mathscr{U}:= L^{\infty}(0,T;H^1)\cap C^{\beta}(0,T;H^{-1}),\quad\mbox{ for some }\beta\in(0,1/2).
\end{align}
Using~\cite[Theorem 5]{Simon1987a}, it is straightforward to deduce that the embedding $\mathscr{U}\hookrightarrow\mathscr{Z}:=C(0,T;L^2)$ is compact. In addition, the sets $G_j:=\{u\in\mathscr{U}:\|u\|_{\mathscr{U}}\leq j\}$ are compact in $\mathscr{Z}$, for each $j\in\mathbb{N}$.  
Now fix $a>0$. If we denote the law of $\rho_{\ep}$ by $\chi_{\ep}$, we get
\begin{align*}
\chi_{\ep}\left(\mathscr{Z}\setminus G_j\right)=\int_{\mathscr{Z}\setminus G_j}{\chi_{\ep}(\m \rho)}=\int_{\mathscr{U}\setminus G_j}{\chi_{\ep}(\m \rho)}\leq \frac{1}{j}\int_{\mathscr{U}}{\|\rho\|_{\mathscr{U}}\chi_{\ep}(\m \rho)}\leq a
\end{align*}
for all $\ep\in(0,1]$, provided that $j$ is large enough, thanks to~\eqref{i:28}. An analogous argument applies to $\{j_{\ep}\}_{\ep}$ and $\{j_{2,\ep}\}_{\ep}$. This corresponds to tightness for the families $\{\rho_{\ep}\}_{\ep}$, $\{j_{\ep}\}_{\ep}$, and $\{j_{2,\ep}\}_{\ep}$, hence the Prokhorov Theorem~\cite[Theorem 14.3]{Kallenberg2002a} is applicable and gives relative compactness in distribution for the three families. In order to complete the proof, we need to show~\eqref{i:28}.

\emph{Uniform bounds for $\{\rho_{\ep}\}_{\ep}$}. We show that 
\begin{subequations}
\label{i:21}
\begin{align}
 \mean{\|\rho_{\ep}\|_{L^{\infty}(0,T;H^1)}} & \leq C,\label{i:21a}\\
 \mean{\|\rho_{\ep}\|_{C^{\beta}(0,T;H^{-1})}} & \leq C,\label{i:21b}
 \end{align}
\end{subequations}
for a constant $C$, independent of $\ep$ and $N$. 
Using~\eqref{i:20a}, we deduce
\begin{align*}
\|\rho_{\ep}\|^2_{L^{\infty}(0,T;H^1)}=\sup_{t\in[0,T]}{\|\rho_{\ep}(\cdot,t)\|^2_{H^1}}\leq 2\|\rho_{\ep}(\cdot,0)\|^2_{H^1}+2T\int_{0}^{T}{\|j_{\ep}(\cdot,s)\|^2_{H^2}\m s}.
\end{align*}
Estimate~\eqref{i:21a} is then settled by invoking Proposition \ref{ip:4}. We now take $v\in H^1$ and compute 
\begin{align}\label{i:30c}
\left|\langle \rho_{\ep}(\cdot,t)-\rho_{\ep}(\cdot,s),v\rangle_{L^2}\right| & = \left|\int_{\mathbb{T}}{\left[\rho_{\ep}(x,t)-\rho_{\ep}(x,s)\right]v(x)\m x}\right|=\left|\int_{\mathbb{T}}{\left(\int_{s}^{t}{-\nabla\cdot j_{\ep}(x,z)\m z}\right)v(x)\m x}\right|\nonumber\\
& = \left|\int_{\mathbb{T}}{\left(\int_{s}^{t}{j_{\ep}(x,z)\m z}\right)\nabla v(x)\m x}\right| \leq \left\|\int_{s}^{t}{j_{\ep}(\cdot,z)\m z}\right\|_{L^2}\|v\|_{H^1}\nonumber\\
& \leq |t-s|^{\frac{1}{2}}\left(\int_{s}^{t}{\left\|j_{\ep}(\cdot,z)\right\|^2_{L^2}\m z}\right)^{\frac{1}{2}}\|v\|_{H^1}.
\end{align}
The bound $x\leq 1+x^2$ valid for any $x\in\mathbb{R}$, the definition of the usual norm of $C^{\beta}(0,T;H^{-1})$, and~\eqref{i:30c} imply
\begin{align}\label{i:31c}
\mean{\|\rho_{\ep}\|_{C^{\beta}(0,T;H^{-1})}}\leq C + C\,\mean{\|\rho_{\ep}(\cdot,0)\|^2_{L^2}+\int_{0}^{T}{\left\|j_{\ep}(\cdot,z)\right\|^2_{L^2}\m z}}\leq C,
\end{align}
for some $\beta\in(0,1/2)$, where the last inequality follows from Proposition \ref{ip:4}. We have thus proved~\eqref{i:21b}.

\emph{Uniform bounds for $\{j_{\ep}\}_{\ep}$}. Again, we show that there exists a constant $C$, independent of $\ep$ and $N$, such that
\begin{subequations}
\label{i:29}
\begin{align}
 \mean{\|j_{\ep}\|_{L^{\infty}(0,T;H^1)}} & \leq C,\label{i:29a}\\
 \mean{\|j_{\ep}\|_{C^{\beta}(0,T;H^{-1})}} & \leq C.\label{i:29b}
\end{align}
\end{subequations}
We use~\eqref{i:20b} and deduce that
\begin{align*}
\|j_{\ep}\|^2_{L^{\infty}(0,T;H^1)} & = \sup_{t\in[0,T]}{\|j_{\ep}(\cdot,t)\|^2_{H^1}}\\
 & \leq C\left\{\|j_{\ep}(\cdot,0)\|^2_{H^1}+\gamma \int_{0}^{T}{\|j_{\ep}(\cdot,z)\|^2_{H^1}\m z}+\int_{0}^{T}{\|j_{2,\ep}(\cdot,z)\|^2_{H^1}\m z}\right.\nonumber\\
& \quad + \int_{0}^{T}{\left\|\frac{1}{N}\sum_{i=1}^{N}{\left(\frac{1}{N}\sum_{j=1}^{N}{W'(q_i(z)-q_j(z))}\right)w_{\ep}(\cdot-q_i(z))}\right\|^2_{H^1}\m z}\nonumber\\
& \left.\quad +\sup_{t\in[0,T]}{\left\|\int_{0}^{t}{\frac{\sigma}{N}\sum_{i=1}^{N}{w_{\ep}(\cdot-q_i(z))\m \beta_i}}\right\|^2_{H^1}}\right\}=:T_1+\dots+T_5.
\end{align*} 
Uniform bounds for $\mean{T_1},\mean{T_2},\mean{T_3}$, and $\mean{T_4}$ are directly given by Proposition \ref{ip:4}. As for $\mean{T_5}$, we invoke~\cite[Theorem 4.36]{Da-Prato2014a} and bound
\begin{align*}
\mean{T_5} & \leq C\,\mean{\int_{0}^{T}{\sum_{i=1}^{N}{\left\|\frac{\sigma}{N}w_{\ep}(\cdot-q_i(s))\right\|^2_{H^1}}\m s}} 
 = C\int_{0}^{T}{\frac{\sigma^2}{N^2}\sum_{i=1}^{N}{\mean{\left\|w_{\ep}(\cdot-q_i(s))\right\|^2_{H^1}}}\m s}\\ 
 & \leq CT\frac{\sigma^2}{N^2}N\left(\frac{1}{\ep}+\frac{1}{\ep^3}\right) \leq \frac{CT\sigma^2}{N\ep^3},
\end{align*}
where the reader is also referred to~\cite[Proof of Proposition 1.1]{Cornalba2018aTR} for the scalings of Sobolev norms of $w_\ep(\cdot-q_i(s))$, which we have used in the second line above. Estimate~\eqref{i:29a} is thus established. In order to prove~\eqref{i:29b}, we analyse the quantity $\left|\langle j_{\ep}(\cdot,t)-j_{\ep}(\cdot,s),v\rangle_{H^{-1},H^1}\right|$. Bounding the relevant contributions coming from the initial datum and the three deterministic integrands is analogous to~\eqref{i:30c}--\eqref{i:31c}. As for the stochastic noise, we rely on~\cite[Lemma 2.1]{Flandoli95} and write, for $\alpha\in(0,1/2)$ and $\lambda>2$ satisfying $\alpha\lambda>1$,
\begin{align*}
& \mean{\left\|\int_{0}^{\cdot}{\frac{\sigma}{N}\sum_{i=1}^{N}{w_{\ep}(\cdot-q_i(t))}\m \beta_i(s)}\right\|^{\lambda}_{W^{\alpha,\lambda}(0,T;H^{-1})}} \\
& \quad \leq C(\alpha,\lambda)\mean{\int_{0}^{T}{\frac{\sigma^{\lambda}}{N^{\lambda}}\left(\sum_{i=1}^{N}{\left\|w_{\ep}(\cdot-q_i(s))\right\|^2_{L^2}}\right)^{\lambda/2}}\m s}\\
& \quad \leq C(\alpha,\lambda)T\frac{\sigma^{\lambda}}{N^{\lambda}}\left(\frac{CN}{\ep}\right)^{\lambda/2}=\frac{C(\alpha,\lambda,\sigma)T}{(N\ep)^{\lambda/2}}. 
\end{align*}
We conclude the analysis for $\mean{T_5}$ using the embedding $W^{\alpha,\lambda}(0,T;H^{-1})\hookrightarrow C^{\beta}(0,T;H^{-1})$ for some $\beta\in(0,\alpha-1/\lambda)$. This embedding is a consequence, e.g., of~\cite{DiBlasio1984}. 
Thus~\eqref{i:29b} is settled. 

\emph{Uniform bounds for $\{j_{2,\ep}\}_{\ep}$}. The argument is almost identical to that used for the family $\{j_{\ep}\}_{\ep}$. We show that 
\begin{subequations}
\label{i:30}
\begin{align}
 \mean{\|j_{2,\ep}\|_{L^{\infty}(0,T;H^1)}} & \leq C,\label{i:30a}\\
 \mean{\|j_{2,\ep}\|_{C^{\beta}(0,T;H^{-1})}} & \leq C,\label{i:30b}
\end{align}
\end{subequations}
for a constant $C$, independent of $\ep,N$. We use~\eqref{i:20c} and deduce that
\begin{align*}
\|j_{2,\ep}\|^2_{L^{\infty}(0,T;H^1)} & = \sup_{t\in[0,T]}{\|j_{2,\ep}(\cdot,t)\|^2_{H^1}} \\
& \leq C\left\{\|j_{2,\ep}(\cdot,0)\|^2_{H^1}+\gamma \int_{0}^{T}{\|j_{2,\ep}(\cdot,z)\|^2_{H^1}\m z}+\int_{0}^{T}{\|j_{3,\ep}(\cdot,z)\|^2_{H^1}\m z}\right.\nonumber\\
& \quad + \int_{0}^{T}{\left\|\frac{1}{N}\sum_{i=1}^{N}{\left(\frac{1}{N}\sum_{j=1}^{N}{W'(q_i(z)-q_j(z))}\right)p_i(t)w'_{\ep}(\cdot-q_i(z))}\right\|^2_{H^1}\m z}\nonumber\\
& \left.\quad +\sup_{t\in[0,T]}{\left\|\int_{0}^{t}{\frac{\sigma}{N}\sum_{i=1}^{N}{p_i(t)w'_{\ep}(\cdot-q_i(z))\m \beta_i}}\right\|^2_{H^1}}\right\}=:T_1+\dots+T_5.
\end{align*} 
The analysis involving the terms $T_1,\dots,T_4$ is analogous to that of the homonyms for $\{j_{\ep}\}_{\ep}$. We only need to deal with the stochastic noise. As for $\mean{T_5}$, 
\begin{align}\label{i:31}
\mean{T_5} & \leq C\,\mean{\int_{0}^{T}{\sum_{i=1}^{N}{\left\|\frac{\sigma}{N}p_i(t)w'_{\ep}(\cdot-q_i(s))\right\|^2_{H^1}}\m s}} 
 = C\int_{0}^{T}{\frac{\sigma^2}{N^2}\sum_{i=1}^{N}{\mean{p_i^2(t)\left\|w'_{\ep}(\cdot-q_i(s))\right\|^2_{H^1}}}}\nonumber\\ 
 & \leq CT\frac{\sigma^2}{N^2}N\mean{p_1^2(t)}\left(\frac{1}{\ep^3}+\frac{1}{\ep^5}\right) \leq \frac{CT\sigma^2}{N\ep^5}.
\end{align}
For $\alpha$ and $\lambda$ as in the previous part of the proof, we use the $\ell^p$-H\"older inequality and bound
\begin{align}\label{i:32}
& \mean{\left\|\int_{0}^{\cdot}{\frac{\sigma}{N}\sum_{i=1}^{N}{p_i(t)w'_{\ep}(\cdot-q_i(t))}\m \beta_i(s)}\right\|^{\lambda}_{W^{\alpha,\lambda}(0,T;H^{-1})}}\nonumber \\
& \quad \leq \frac{C(\alpha,\lambda)\sigma^{\lambda}}{N^{\lambda}}\int_{0}^{T}{\mean{\left(\sum_{i=1}^{N}{\left\|p_i(s)w'_{\ep}(\cdot-q_i(s))\right\|^2_{L^2}}\right)^{\lambda/2}}\m s} \nonumber\\
& \quad \leq \frac{C(\alpha,\lambda)\sigma^{\lambda}}{N^{\lambda}\ep^{3\lambda/2}}\int_{0}^{T}{\mean{\left(\sum_{i=1}^{N}{p^2_i(s)}\right)^{\lambda/2}}\m s}\nonumber \\
& \quad \leq \frac{C(\alpha,\lambda)\sigma^{\lambda}}{N^{\lambda}\ep^{3\lambda/2}}\int_{0}^{T}{N^{\lambda/2-1}\mean{\left(\sum_{i=1}^{N}{p^{\lambda}_i(s)}\right)}\m s} = \frac{C(\alpha,\lambda)\sigma^{\lambda}}{N^{\lambda/2}\ep^{3\lambda/2}}\int_{0}^{T}{\mean{p^{\lambda}_1(s)}\m s} = \frac{C(\alpha,\lambda,T)\sigma^{\lambda}}{N^{\lambda/2}\ep^{3\lambda/2}}.
\end{align}
Inequalities~\eqref{i:31} and~\eqref{i:32} allow us to deduce~\eqref{i:30a} and~\eqref{i:30b}, and the proof is complete. 
\end{proof}
\begin{rem}\label{irem:5}
In contrast to the methodology employed in~\cite[Proposition 1.1]{Cornalba2018aTR}, which settles tightness in the case of independent particles, the proof of Proposition \ref{ip:3} does not rely on the Kolmogorov criterion. The reason is that the time regularity associated with the application of the propagation of chaos is not sufficiently high.
\end{rem}
\begin{rem}
In principle, there is more than one natural choice for the definition of the space $\mathscr{U}$. Specifically, in~\eqref{i:4000}, one might replace $H^{-1}$ with any $H^{-k}$, where $k\in\mathbb{N}\cup\{0\}$, thus including $L^2$. This would result in adapting estimate~\eqref{i:30c} in the case of $\{\rho_{\ep}\}_{\ep}$ (and analogous expressions in the case of $\{j_{\ep}\}_{\ep}$ and $\{j_{2,\ep}\}_{\ep}$), thus invoking Proposition \ref{ip:4} with a different parameter $n$. This directly reflects in a possibly different requirement for the scaling $N\ep^{\theta}=1$. Since we are not concerned with the lowest admissible value of $\theta$, the choice of $H^{-1}$ is as good as any other of those listed above.
\end{rem}

\subsection{Approximating the interaction term}\label{ss:3.3}

We show that the third term of the right-hand-side of~\eqref{i:20b} is asymptotically equivalent (in the limit $\ep\rightarrow 0$ and $N\rightarrow 0$) to the nonlocal interaction term $\{W'\ast \rho_\ep\}\rho_{\ep}$.

\begin{prop}\label{ip:6}
Let $T>0$. Let the assumptions of Propositions \ref{ip:1} and \ref{ip:2} be satisfied. Assume the scaling $N\ep^{\theta}=1$, for $\theta$ large enough. We have the equality
\begin{align}\label{i:33}
\frac{1}{N}\sum_{i=1}^{N}{\left(\frac{1}{N}\sum_{j=1}^{N}{W'(q_i(t)-q_j(t))}\right)w_{\ep}(x-q_i(t))}=\left\{W'\ast\rho_{\ep}(\cdot,t)\right\}(x)\rho_\ep(x,t)+r_{1,\ep}\rho_{\ep}(x,t)+r_{2,\ep},
\end{align}
where $r_1$ and $r_2$ 
are stochastic remainders such that $|r_{1,\ep}|\leq C(W)\sqrt{\ep}$ and $\mean{|r_{2,\ep}|}\leq C(W,f_0)\{\sqrt{\ep}+\ep^{\beta}\}$,  
for some $\beta=\beta(\theta)>0$, and where $f_0$ is as in Proposition \ref{ip:2}.
\end{prop}
Before we prove the result above, we recall a simple lemma.
\begin{lemma}\label{ilem:3}
Let $f\in \mathcal{C}^0(\mathbb{T})$ be a Lipschitz function. There is a constant $C=C(f)$, independent of $\ep>0$ and $a\in\mathbb{T}$, such that
$
\left|\int_{\mathbb{T}}{w_{\ep}(y-a)f(y)\emph{\m} y}-f(a)\right|\leq C\left(\sqrt{\ep}+\exp\left\{-C\ep^{-1}\right\}\right).
$
\end{lemma}
\begin{proof}
Let $A_\ep:=(a-\sqrt{\ep},a+\sqrt{\ep})$. Since $f$ is Lipschitz, we obtain
\begin{align}\label{i:100}
\int_{\mathbb{T}}{w_{\ep}(y-a)f(y)\m y} & =\int_{A_{\ep}}{w_{\ep}(y-a)f(y)\m y}+\int_{\mathbb{T}\setminus A_{\ep}}{w_{\ep}(y-a)f(y)\m y} \nonumber\\
& \geq (f(a)-C\sqrt{\ep})\int_{A_{\ep}}{w_{\ep}(y-a)\m y} + \min_{x\in\mathbb{T}}{f}\int_{\mathbb{T}\setminus A_{\ep}}{w_{\ep}(y-a)\m y} \nonumber\\
& \geq f(a)\left(1-\int_{\mathbb{T}\setminus A_{\ep}}{w_{\ep}(y-a)\m y}\right)-C\sqrt{\ep} + \min_{x\in\mathbb{T}}{f}\int_{\mathbb{T}\setminus A_{\ep}}{w_{\ep}(y-a)\m y}.
\end{align}
It is immediate to notice that $\int_{\mathbb{T}\setminus A_{\ep}}{w_{\ep}(y-a)\m y}\leq C\exp\left\{-C\ep^{-1}\right\}$ for some $C>0$. From~\eqref{i:100}, we obtain
\begin{align*}
\int_{\mathbb{T}}{w_{\ep}(y-a)f(y)\m y}-f(a) & \geq -f(a)\int_{\mathbb{T}\setminus A_{\ep}}{w_{\ep}(y-a)\m y} +\min_{x\in\mathbb{T}}{f}\int_{\mathbb{T}\setminus A_{\ep}}{w_{\ep}(y-a)\m y} -C\sqrt{\ep} \\
& \geq C\left(\min_{x\in\mathbb{T}}{f}-\max_{x\in\mathbb{T}}{f}\right)\exp\left\{-C\ep^{-1}\right\}-C\sqrt{\ep}.
\end{align*}
An analogous inequality (with opposite sign) may be obtained in a similar way, completing the proof. 
\end{proof}
\begin{proof}[Proof of Proposition \ref{ip:6}]
We split the left-hand-side of~\eqref{i:33} as $T_1+T_2$, where
\begin{subequations}
\begin{align*}
T_1 & := \frac{1}{N}\sum_{i=1}^{N}{\left(\frac{1}{N}\sum_{j=1}^{N}{W'(x-q_j(t))}\right)w_{\ep}(x-q_i(t))}\intertext{and}
T_2 & := \frac{1}{N}\sum_{i=1}^{N}{\left(\frac{1}{N}\sum_{j=1}^{N}{\left\{W'(q_i(t)-q_j(t))-W'(x-q_j(t))\right\}}\right)w_{\ep}(x-q_i(t))}.
\end{align*}
\end{subequations}
As for $T_1$, we separate the sums in $i$ and $j$ and deduce
\begin{align*}
T_1 & = \left(\frac{1}{N}\sum_{i=1}^{N}{w_{\ep}(x-q_i(t))}\right)\left(\frac{1}{N}\sum_{j=1}^{N}{W'(x-q_j(t))}\right)=\rho_{\ep}(x,t)\left(\frac{1}{N}\sum_{j=1}^{N}{W'(x-q_j(t))}\right)\\
& = \rho_{\ep}(x,t)\left(r_{1,\ep}+\frac{1}{N}\sum_{j=1}^{N}{\int_{\mathbb{T}}{W'(x-y)w_{\ep}(y-q_j(t))\m y}}\right)\\
& =\left\{W'\ast\rho_{\ep}(\cdot,t)\right\}\rho_\ep(x,t)+r_{1,\ep}\rho_{\ep}(x,t).
\end{align*} 
Lemma \ref{ilem:3} gives $|r_{1,\ep}|\leq C(W)\sqrt{\ep}$, where $C$ is independent of $x,t,\omega$.
With the notation of~\eqref{i:33}, it holds that $r_{2,\ep}=T_2$. We use a Taylor expansion and bound
\begin{align*}
|r_{2,\ep}| & \leq \frac{1}{N^2}\sum_{i,j=1}^{N}{\left|W'(q_i(t)-q_j(t))-W'(x-q_j(t))\right|w_{\ep}(x-q_i(t))}\\
& \leq \frac{C(W)}{N^2}\sum_{i,j=1}^{N}{\left|x-q_i(t)\right|w_{\ep}(x-q_i(t))}=\frac{C(W)}{N}\sum_{i=1}^{N}{\left|x-q_i(t)\right|w_{\ep}(x-q_i(t))}\\
& =  \frac{C(W)}{N}\sum_{i=1}^{N}{\left|x-\ov{q}_i(t)\right|w_{\ep}(x-\ov{q}_i(t))} \\
& \quad + \frac{C(W)}{N}\sum_{i=1}^{N}{\left\{\left|x-q_i(t)\right|w_{\ep}(x-q_i(t))-\left|x-\ov{q}_i(t)\right|w_{\ep}(x-\ov{q}_i(t))\right\}} =: T_{3}+T_4.
\end{align*} 
Since the particles are identically distributed, we have
\begin{align*}
\mean{|T_3|} & =C(W)\mean{\left|x-q_1(t)\right|w_{\ep}(x-q_1(t))} \\
& = C(W)\int_{\mathbb{T}}{\left|y-x\right|w_{\ep}(x-y)f_{\ov{q}}(t,y)\m y} \leq C(W,f_0)\sqrt{\ep},
\end{align*} 
where $f_{\ov{q}}(t,\cdot)$ is the probability density function of $\ov{q}(t)$, and $f_0$ is as in Proposition \ref{ip:2}. The last inequality above is given by Lemma \ref{ilem:3}: in particular, the constant $C$ does not depend on time, as $\sup_{t\geq 0,q\in\mathbb{T}}{\frac{\partial}{\partial q}f_{\ov{q}}(t,q)}$ is finite. To see this, one may apply~\citep[(17.2)]{Villani2009b} and~\citep[Theorem 4.12]{Adams2003a}, with analogous considerations to those made in the proof of Proposition \ref{ip:2}.

 As for $T_4$, we use a Taylor expansion, the bounds $\max_{x\in\mathbb{T}}{w_{\ep}(x)}\leq C\ep^{-1}$ and $\max_{x\in\mathbb{T}}{|w'_{\ep}(x)|}\leq C\ep^{-2}$, and write
\begin{align*}
\mean{|T_4|} & =  C(W)\mean{\left|x-q_1(t)\right|w_{\ep}(x-q_1(t))-\left|x-\ov{q}_1(t)\right|w_{\ep}(x-\ov{q}_1(t))}\\
& \leq C(W)\mean{\left|x-q_1(t)\right|\cdot\left|w_{\ep}(x-q_1(t))-w_{\ep}(x-\ov{q}_1(t))\right|}\\
& \quad + C(W)\mean{|q_1(t)-\ov{q}_1(t)|w_{\ep}(x-\ov{q}_1(t))}\\
& \leq C(W)\ep^{-2}\mean{\left|x-q_1(t)\right|\cdot\left|q_1(t)-\ov{q}_1(t)\right|}\\
& \quad + C(W)\ep^{-1}\mean{|q_1(t)-\ov{q}_1(t)|}\leq C(W)\ep^{\beta},
\end{align*} 
for some $\beta=\beta(\theta)>0$, where the last inequality follows from the propagation of chaos (Proposition \ref{ip:1}), and the scaling $N\ep^{\theta}=1$. The bound for $r_{2,\ep}$ is established, and the proof is complete.
\end{proof}

\subsection{Noise comparison}\label{ss:3.4}

We want to replace the stochastic noise of~\eqref{i:20b} (previously referred to as $\mathcal{Z}_N$) with a noise closed in $\rho_{\ep}$ and $j_{\ep}$. We suitably adapt~\cite[Subsections 3.2 and 3.3]{Cornalba2018aTR}. 

We first recall a useful fact. Let $\gamma_{\ep}$ be the probability density function of a Gaussian random variable with mean zero and variance $\ep^2$. It is not difficult to show that, for $r_{\ep}:=w_{\ep}-\gamma_{\ep}$, it holds that \begin{equation}\label{i:111}
\|r_\ep\|_{C^0(-\pi;\pi)}\leq {\ep}^{\alpha},\mbox{ for some }\alpha\in(0,1).
\end{equation} 

\begin{prop}\label{ip:11}
Let the assumptions of Propositions \ref{ip:1} and \ref{ip:2} be satisfied. Assume the scaling $N\ep^{\theta}=1$, for $\theta$ large enough. We define the stochastic noise $$\mathcal{\dot{Y}}_N:=\sigma N^{-1/2}\sqrt{\rho_{\ep/\sqrt{2}}}\,Q^{1/2}_{\sqrt{2}\ep}\xi,$$ where $\xi$ is space-time white noise and $Q_{\sqrt{2}\ep}\colon L^2\rightarrow L^2$ is the convolution operator with kernel $w_{\sqrt{2}\ep}$ (i.e., $\tilde{\xi}_\ep:=Q^{1/2}_{\sqrt{2}\ep}\xi$ is an $H^1$-valued Q-Wiener process with covariance operator $Q_{\sqrt{2}\ep}$). For some positive $C=C(T)$, $c_1(\theta)$, and $c_2(\theta)$, and $\alpha$ as in~\eqref{i:111}, we have 
\begin{align*}
& \left|\mean{\mathcal{Z}_N(x_1,t)\mathcal{Z}_N(x_2,t)}-\mean{\mathcal{Y}_N(x_1,t)\mathcal{Y}_N(x_2,t)}\right| \\
& \quad \leq\frac{C\sigma^2}{N}w_{\sqrt{2}\ep}(x_1-x_2)\times\left\{|x_1-x_2|+\ep^{c_1(\theta)}+\ep^{\alpha}+\ep^{c_2(\theta)}|x_1-x_2|^{1/2}\right\}+ \frac{C\sigma^2}{N}\ep^{\alpha}.
\end{align*}
\end{prop}
This result is an adaptation of~\cite[Proof of Theorem 1.3]{Cornalba2018aTR}. We sketch the proof below, and defer more technical considerations to Remark \ref{rem:100}.
\begin{proof}[Proof of Proposition \ref{ip:11}]
In what follows, the residuals $r_{\ep}$ in~\eqref{i:111} appear several times. We do not specify the argument, as ultimately only their $C^0$-norms will play a role. Set $m:=(x_1+x_2)/2$. We use the multiplication rule for Gaussian kernels~\cite[Lemma A.4]{Cornalba2018aTR}, the independence of the Brownian noises, and we apply~\eqref{i:111} several times to obtain 
\begin{align*}
  &\mean{\mathcal{Z}_{N}(x_1,t)\mathcal{Z}_{N}(x_2,t)} \\
  &\quad = \mean{\left(\int_{0}^{t}{\frac{\sigma}{N}\sum_{i=1}^{N}{w_{\ep}(x_1-q_i(u))
    \m \beta_i(u)}}\right)\left(\int_{0}^{t}{\frac{\sigma}{N}\sum_{i=1}^{N}{w_{\ep}(x_2-q_i(u))\m \beta_i(u)}}\right)}\nonumber\\
  &\quad = \frac{\sigma^2}{N^2}\mean{\sum_{i=1}^{N}{\int_{0}^{t}{w_{\ep}(x_1-q_i(u))w_{\ep}(x_2-q_i(u))\m u}}}\\
  &\quad = \frac{\sigma^2}{N^2}\mean{\sum_{i=1}^{N}{\int_{0}^{t}{\left(\gamma_{\ep}(x_1-q_i(u))+r_{\ep}\right)\left(\gamma_{\ep}(x_2-q_i(u))+r_{\ep}\right)\m u}}}\\
  &\quad = \frac{\sigma^2}{N^2}\mean{\sum_{i=1}^{N}{\int_{0}^{t}{\gamma_{\sqrt{2}\ep}(x_1-x_2)\gamma_{\ep/\sqrt{2}}(m-q_i(u))\m u}}}\\
  & \quad \quad + \frac{\sigma^2}{N^2}\mean{\sum_{i=1}^{N}{\int_{0}^{t}{\left\{r_{\ep}^2+r_{\ep}\gamma_{\ep}(x_1-q_i(u))+r_{\ep}\gamma_{\ep}(x_2-q_i(u))\right\}\m u}}}.
\end{align*}
We use~\eqref{i:111} to switch back to the von Mises kernels, and use the definition of $\rho_{\ep/\sqrt{2}}$ to obtain
\begin{align*}
  &\left|\mean{\mathcal{Z}_{N}(x_1,t)\mathcal{Z}_{N}(x_2,t)} - \mean{\mathcal{Y}_N(x_1,t)\mathcal{Y}_N(x_2,t)}\right| \\
     & \quad \leq \left|\frac{\sigma^2}{N}w_{\sqrt{2}\ep}(x_1-x_2)\int_{0}^{t}{\mean{\rho_{\ep/\sqrt{2}}(m,u)}\m u} \right. \\ 
     & \quad \quad \left. - \frac{\sigma^2}{N}w_{\sqrt{2}\ep}(x_1-x_2)\int_{0}^{t}{\mean{\sqrt{\rho_{\ep/\sqrt{2}}(x_1,u)\rho_{\ep/\sqrt{2}}(x_2,u)}}\m u} \right|
  \\
  & \quad \quad + \left|\frac{\sigma^2}{N^2}\mean{\sum_{i=1}^{N}{\int_{0}^{t}{\left\{3r_\ep^2+r_{\ep/\sqrt{2}}r_{\sqrt{2}\ep}\right\}\m u}}} + \frac{\sigma^2}{N^2}\mean{\sum_{i=1}^{N}{\int_{0}^{t}{\left\{r_{\ep/\sqrt{2}}w_{\sqrt{2}\ep}(x_1-x_2)\right\}\m u}}}\right.\\
  & \quad \quad + \left.\frac{\sigma^2}{N^2}\mean{\sum_{i=1}^{N}{\int_{0}^{t}{\left\{r_{\ep}w_{\ep}(x_1-q_i(u))+r_{\ep}w_{\ep}(x_2-q_i(u))+r_{\sqrt{2}\ep}w_{\ep/\sqrt{2}}(m-q_i(u))\right\}\m u}}}\right|\\
  & \quad =: |A_1-A_2|+|A_3+A_4+A_5|.
\end{align*}
The bound $|A_3+A_4+A_5|\leq (C\sigma^2/N)\{\ep^\alpha+\ep^{\alpha}w_{\sqrt{2}\ep}(x_1-x_2)\}$ follows easily from~\eqref{i:111}. In order to control $|A_1-A_2|$, it is sufficient to bound 
\begin{align}\label{i:1003}
\mean{\left|\rho_{\ep/\sqrt{2}}(m)-\sqrt{\rho_{\ep/\sqrt{2}}(x_1)\rho_{\ep/\sqrt{2}}(x_2)}\right|}, 
\end{align}
where we have fixed $u\in[0,T]$, and dropped the time dependence for notational convenience. We bound the random variable in~\eqref{i:1003} as
\begin{align}\label{i:1004}
\left|\rho_{\ep/\sqrt{2}}(m)-\sqrt{\rho^2_{\ep/\sqrt{2}}(m)+b(x_1,x_2)}\right|\leq \sqrt{\left|b(x_1,x_2)\right|},
\end{align}
where
\begin{align*}
b(x_1,x_2) & := \rho_{\ep/\sqrt{2}}(m)\!\left[\rho_{\ep/\sqrt{2}}(x_1)+\rho_{\ep/\sqrt{2}}(x_2)-2\rho_{\ep/\sqrt{2}}(m)\right]\\ 
& \quad + (\rho_{\ep/\sqrt{2}}(x_1)-\rho_{\ep/\sqrt{2}}(m))(\rho_{\ep/\sqrt{2}}(x_2)-\rho_{\ep/\sqrt{2}}(m)).
\end{align*}
The H\"older inequality implies that $\mean{\sqrt{\left|b(x_1,x_2)\right|}}$ is bounded by 
\begin{align*}
& \mean{\rho^2_{\ep/\sqrt{2}}(m)}^{1/4}\mean{\left|\rho_{\ep/\sqrt{2}}(x_1)+\rho_{\ep/\sqrt{2}}(x_2)-2\rho_{\ep/\sqrt{2}}(m)\right|^2}^{1/4}\\
& \quad + \mean{\left|\rho_{\ep/\sqrt{2}}(x_1)-\rho_{\ep/\sqrt{2}}(m)\right|^4}^{1/8}\mean{\left|\rho_{\ep/\sqrt{2}}(x_2)-\rho_{\ep/\sqrt{2}}(m)\right|^4}^{1/8}=:T_1T_2+T_3T_4.
\end{align*}
We notice that
\begin{align}\label{i:2002}
& \mean{\left|\frac{1}{N}\sum_{i=1}{w_{\ep}(x-\ov{q}_i(t))}\right|^{c}} \nonumber\\
& \quad = N^{-c}\sum_{\mathbf{j}\in\mathcal{S}_{1,c}}{\mean{\prod_{k=1}^{c}{w_{\ep}(x-\ov{q}_{j_k}(t))}}} + N^{-c}\sum_{\mathbf{j}\in\mathcal{S}_{2,c}}{\mean{\prod_{k=1}^{c}{w_{\ep}(x-\ov{q}_{j_k}(t))}}}\nonumber\\
& \quad \leq \mean{w_{\ep}(x-\ov{q}_1(t))}^c+N^{-1}\ep^{-c} \leq \|w_{\ep}(x-\cdot)\|_{L^1}^c\|f_{\ov{q}}(t,\cdot)\|^c_{L^\infty}+N^{-1}\ep^{-c}\nonumber\\
& \quad = \|f_{\ov{q}}(t,\cdot)\|^c_{L^\infty}+N^{-1}\ep^{-c},
\end{align}
where $f_{\ov{q}}(t,\cdot)$ is the probability density function of $\ov{q}(t)$, and $f_0$ is as in Proposition \ref{ip:2}. As $\theta$ is large enough, and taking into account $\sup_{t\geq0}\|f_{\ov{q}}(t,\cdot)\|^c_{L^\infty}<\infty$ (implied by assumptions of Proposition \ref{ip:2} thanks to~\citep[(17.2)]{Villani2009b}) we see that the left-hand side of~\eqref{i:2002} is uniformly bounded in $\ep,x$, and $t$. We may now bound $T_1,\dots,T_4$. We write
\begin{align*}
T_1 \leq K\mean{\ov{\rho}^2_{\ep/\sqrt{2}}(m)}^{1/4}+K\mean{\left|\rho_{\ep/\sqrt{2}}(m)-\ov{\rho}_{\ep/\sqrt{2}}(m)\right|^2}^{1/4},
\end{align*}
where $\ov{\rho}_{\ep}$ is the smoothed density with respect to the particle system~\eqref{i:2}. The first term in the right-hand side above is bounded by~\eqref{i:2002}, while the second is bounded using the propagation of chaos. As a result, $T_1\leq C$.

As for $T_2$, again by adding and subtracting relevant evaluations of $\ov{\rho}_{\ep}$, we obtain
\begin{align}\label{i:2003}
T_2 & \leq K\mean{\left|\ov{\rho}_{\ep/\sqrt{2}}(x_1)+\ov{\rho}_{\ep/\sqrt{2}}(x_2)-2\ov{\rho}_{\ep/\sqrt{2}}(m)\right|^2}^{1/4}+K\mean{\left|\rho_{\ep/\sqrt{2}}(x_1)-\ov{\rho}_{\ep/\sqrt{2}}(x_1)\right|^2}^{1/4}\nonumber\\
& \quad +K\mean{\left|\rho_{\ep/\sqrt{2}}(x_2)-\ov{\rho}_{\ep/\sqrt{2}}(x_2)\right|^2}^{1/4} +K\mean{\left|\rho_{\ep/\sqrt{2}}(m)-\ov{\rho}_{\ep/\sqrt{2}}(m)\right|^2}^{1/4}.
\end{align}
The first term in the right-hand side of~\eqref{i:2003} can bounded by $K|x_1-x_2|$, using the same strategy used in~\cite[Adaptation of proof of Theorem 1.3]{Cornalba2018aTR}; the remaining ones are controlled using the propagation of chaos. As a result, we get $T_2\leq K|x_1-x_2|+\ep^{\gamma_1}$, for some $\gamma_1=\gamma_1(\theta)>0$.
The analysis of $T_3,T_4$ is similar to that of $T_2$. In the case of $T_3$ 
\begin{align}\label{i:2004}
T_3 & \leq K\mean{\left|\ov{\rho}_{\ep/\sqrt{2}}(x_1)-\ov{\rho}_{\ep/\sqrt{2}}(m)\right|^4}^{1/8}+K\mean{\left|\rho_{\ep/\sqrt{2}}(x_1)-\ov{\rho}_{\ep/\sqrt{2}}(x_1)\right|^2}^{1/4}\nonumber\\
& \quad +K\mean{\left|\rho_{\ep/\sqrt{2}}(m)-\ov{\rho}_{\ep/\sqrt{2}}(m)\right|^2}^{1/4}.
\end{align}
The first term in the right-hand side of~\eqref{i:2004} can bounded by $K\sqrt{|x_1-x_2|}$, using the same strategy used in~\cite[Adaptation of proof of Theorem 1.3]{Cornalba2018aTR}; propagation of chaos controls the remaining ones. So $T_3\leq K\sqrt{|x_1-x_2|}+\ep^{\gamma_2}$, for some $\gamma_2=\gamma_2(\theta)>0$. The estimate for $T_4$ is the same, with the couple $(x_1,m)$ replaced by $(x_2,m)$.

Putting everything together, we obtain the bound 
\begin{align}\label{i:2005}
\mean{\left|\rho_{\ep/\sqrt{2}}(m)-\sqrt{\rho^2_{\ep/\sqrt{2}}(m)-b(x_1,x_2)}\right|}\leq C\left\{|x_1-x_2|+\ep^{c_1(\theta)}+\ep^{c_2(\theta)}|x_1-x_2|^{1/2}\right\},
\end{align}
where $c_1(\theta):=\min\{\gamma_1;2\gamma_2\}$ and $c_2(\theta):=\gamma_2$. This concludes the proof.
\end{proof}
\begin{rem}\label{rem:100}
The error bound of Proposition \ref{ip:11} is less sharp than the one provided in~\cite[Theorem~1.3]{Cornalba2018aTR} in the following sense: firstly, the spatial term contributions in~\eqref{i:2005} are not quadratic. This is due to the use of the suboptimal bound~\eqref{i:1004}, as clarified in~\cite[Remark 3.4]{Cornalba2018aTR}. More precisely, we do not have an analogue of~\cite[Proposition B.8]{Cornalba2018aTR} in the case of weakly interacting particles, so we can not use more precise bounds involving inverse powers of $\rho_{\ep}$; secondly, the propagation of chaos produces stand-alone contributions in $\ep$ (vanishing as $\ep\rightarrow 0$); finally, the need to switch from von Mises to Gaussian kernels (and vice versa) produces additional contributions in $\ep$ (also vanishing as $\ep\rightarrow 0$).
\end{rem}

\section{The regularised model}\label{s:4}

While the equations~\eqref{i:20a}--\eqref{i:20b} describe the `exact' evolution of the relevant densities $(\rho_\ep,j_{\ep})$ associated to the weakly interacting particle system~\eqref{i:1}, they are not, however, closable in $(\rho_{\ep},j_{\ep})$: more precisely, they contain three terms (specifically, $j_{2,\ep}$, $\mathcal{Z}_N$, and the nonlocal interaction term of~\eqref{i:20b}) which can not be related directly to $(\rho_{\ep},j_{\ep})$. In this final section, under suitable assumptions, we derive and analyse an SPDE 
which approximates~\eqref{i:20a}--\eqref{i:20b}. We propose the following approximations associated with the three terms mentioned above, and we point out the extent to which they are valid.

\emph{Approximation 1}. The interaction term in~\eqref{i:20b} is replaced by $\{W'\ast \rho_\ep\}\rho_{\ep}$. Proposition \ref{ip:6} implies that this replacement gives a vanishing error (in the $L^1$ sense) as $\ep\rightarrow 0$.

\emph{Approximation 2}.  We replace $j_{2,\ep}$ with $\frac{\sigma^2}{2\gamma}\frac{\partial\rho_{\ep}}{\partial x}$. This has been done also in~\cite{Cornalba2018aTR}, and we adapt the essential details here. In local equilibrium, the probability density function of the couple $(q_i(t),p_i(t))$ is approximately separable in the two variables (as shown in~\cite[Corollary 3.2]{Duong2018a}). We can thus write $\mean{j_{2,\ep}}=\mean{p_1^2(t)}\mean{\partial \rho_{\ep}/\partial x}$, which suggests the proposed replacement. In a small temperature regime (corresponding to $\sigma^2/(2\gamma)\ll 1$), we see that $\mbox{Var}[p^2_i(t)]\leq C\sigma^4/(2\gamma)^2\ll \sigma^2/(2\gamma)\approx \mean{p^2_i(t)}$, see again~\cite[Corollary 3.2]{Duong2018a}. It is in this case sensible to replace $p^2_i$ with $\mean{p^2_i}$, which means replacing $j_{2,\ep}$ with $\frac{\sigma^2}{2\gamma}\frac{\partial\rho_{\ep}}{\partial x}$.

\emph{Approximation 3}. We replace $\mathcal{Z}_N$ with $\sigma N^{-1/2}\sqrt{\rho_{\ep}}\,\tilde{\xi_{\ep}}$. This is justified along the lines of~\cite{Cornalba2018aTR}, and we adapt the necessary details. First, we notice that $\mathcal{Z}_N$ and $\mathcal{Y}_N$ are asymptotically equivalent in distribution for $\ep\rightarrow 0$, as shown in Proposition \ref{ip:11}. In addition, one can show that, for each $t\in[0,T]$, $\{\rho_{\ep}(\cdot,t)\}_{\ep}$ has a unique limit in $L^2$ as $\ep\rightarrow 0$. This can be seen be taking two sequences $\{a_n;N^a_n\}$, $\{b_n;N^b_n\}$ (both satisfying the usual $\theta$-scaling) and using scaling arguments (similar to those used, for example, in~\eqref{i:2002}) and the propagation of chaos to show that $\mean{\|\rho_{a_n}(\cdot,t)-\rho_{b_n}(\cdot,t)\|^2_{L^2}}\rightarrow 0$ as $a_n, b_n\rightarrow 0$. As a result, the two quantities $\rho_{\ep}(\cdot,t)$ and $\rho_{\ep/\sqrt{2}}(\cdot,t)$ coincide in the limit. Therefore, for $\ep\ll 1$, we consider $\sigma N^{-1/2}\sqrt{\rho_{\ep}}\,\tilde{\xi_{\ep}}$ in spite of  $\mathcal{Y}_N$, thus obtaining the overall noise replacement.

These approximations give the following \emph{regularised Dean--Kawasaki model} for interacting particles in undamped regime
\begin{subequations}
\label{i:300}
\begin{empheq}[left={}\empheqlbrace]{align}
  \,\,\displaystyle\frac{\partial \tilde{\rho}_{\ep}}{\partial t}(x,t) & = -\frac{\partial \tilde{j}_{\ep}}{\partial x}(x,t),\label{i:300a}\\
  \,\,\displaystyle\frac{\partial  \tilde{j}_{\ep}}{\partial t}(x,t) & = 
  -\gamma  \tilde{j}_{\ep}(x,t)-\left(\frac{\sigma^2}{2\gamma}\right)\frac{\partial  \tilde{\rho}_{\ep}}{\partial x}(x,t)-\{W'\ast \tilde{\rho}_{\ep}(\cdot,t)\} \tilde{\rho}_{\ep}(\cdot,t)
  +\frac{\sigma}{\sqrt{N}}\sqrt{ \tilde{\rho}_{\ep}(x,t)}\,\tilde{\xi}_{\ep}, \label{i:300b}\\
  \,\, \displaystyle \tilde{\rho}_{\ep}(x,0) & = \rho_0(x),\quad  \tilde{j}_{\ep}(x,0)= j_0(x)\nonumber,
  \end{empheq}
\end{subequations}
for $ (x,t)\in\mathbb{T}\times[0,T]$, and where $(\rho_0,j_0)$ is a suitable initial datum. We used the notation $(\tilde{\rho}_{\ep},\tilde{j}_{\ep})$ to distinguish the solution of the SPDE~\eqref{i:300} from the smoothed (exact) densities $(\rho_{\ep},j_{\ep})$. We establish a high-probability existence and uniqueness result (in the sense of mild solutions) for~\eqref{i:300}. 
Following~\cite[Subsection 4.3]{Cornalba2018aTR}, we smooth the coefficient function of the noise in~\eqref{i:300b} and study the system
\begin{equation}
  \label{i:301}
  \left\{
    \begin{array}{l}
      \m X_{\ep}(t)=\left[AX_{\ep}(t)+\alpha(X_{\ep}(t))\right]\m t+B_{N,\delta}(X_{\ep}(t))\m W_{\ep}, \\
      X_{\ep}(0)=X_0,
    \end{array}
  \right.
\end{equation}
for $X_{\ep}(t):=(\tilde{\rho}_{\ep}(\cdot,t),\tilde{j}_{\ep}(\cdot,t))$, $X_0:=(\rho_0,j_0)$, $\dot{W}_{\ep}:=(0,\tilde{\xi}_{\ep})$, and where $A$ (respectively, $\alpha$) is a linear (respectively, nonlinear) operator on $\mathcal{W}:=H^1(\mathbb{T})\times H^1(\mathbb{T})$ defined by 
\begin{equation*}
  AX_{\ep}(t):=\left(-\frac{\partial \tilde{j}_{\ep}}{\partial x}(\cdot,t),\,-\gamma \tilde{j}_{\ep}(\cdot,t)-\left(\frac{\sigma^2}{2\gamma}\right)\frac{\partial \tilde{\rho}_{\ep}}{\partial x}(\cdot,t)\right),\qquad 
  \alpha (X_{\ep}(t)):=\left(0,-\{W'\ast\tilde{\rho}_{\ep}(\cdot,t)\}\tilde{\rho}_{\ep}(\cdot,t)\right),
\end{equation*}
 and $B_{N,\delta}\colon\mathcal{W}\rightarrow \{f\colon \mathcal{W}\rightarrow L^2\times L^2\}$ is defined as $B_{N}((\rho,j))(a,b):=\sigma N^{-1/2}\left(0,\, h_{\delta}(|\rho|)\cdot b\right)$, for $h_{\delta}$ being a $C^2(\mathbb{R})$-regularisation of the square-root function on $[-\delta,\delta]$, for some $\delta>0$. A mild solution to~\eqref{i:301} on $[0,T]$ is a $\mathcal{W}$-valued predictable process $X_{\ep,\delta}=(\tilde{\rho}_{\ep,\delta},\tilde{j}_{\ep,\delta})$ defined on $[0,T]$ such that $\mathbb{P}(\int_{0}^{T}{\|X_{\ep,\delta}(s)\|^2_{\mathcal{W}}\m s})=1$, and satisfying, for each $t\in[0,T]$
 \begin{align*}
 X_{\ep,\delta}(t)=S(t)X_0+\int_{0}^{t}{S(t-s)\alpha(X_{\ep,\delta}(s))\m s}+\int_{0}^{t}{S(t-s)B_{N,\delta}(X_{\ep,\delta}(s))\m W_{\ep}},\qquad\mathbb{P}\mbox{-a.s.}
 \end{align*}
 where $\{S(t)\}_{t\geq 0}$ is the $C_0$-semigroup generated by $A$ (see~\cite[Lemma 4.2]{Cornalba2018aTR}).

We first of all analyse the noise-free version of~\eqref{i:301}.
\begin{lemma}\label{il:7} Fix $0<c_1<c_2$. Consider the system
\begin{equation}
  \label{i:302}
  \left\{
    \begin{array}{l}
      \emph{\m} X(t)=\left[AX(t)+\alpha(X(t))\right]\emph{\m} t,\\
      X(0)=X_0:=(\rho_0,j_0),
    \end{array}
  \right.
\end{equation}
and assume that $\min_{x\in\mathbb{T}}{\rho_0(x)}> c_1$ and $\|X_0\|_{\mathcal{W}}< c_2$. Then~\eqref{i:302} has a unique local $\mathcal{W}$-valued mild solution $Z:=(\rho_Z,j_Z)$ up to some $T>0$, such that 
\begin{align}\label{i:303}
\min_{x\in\mathbb{T}}{\rho_Z(x,s)}> c_1\mbox{ and }\|Z(s)\|_{\mathcal{W}}< c_2,\quad\mbox{ for all }s\in[0,T].
\end{align}
\end{lemma}
\begin{proof}
The operator $A$ generates a $C_0$-semigroup of contractions on $\mathcal{W}$, see for example~\cite[Lemma 4.2]{Cornalba2018aTR}. In addition, $\alpha$ is locally Lipschitz and locally bounded. To see this, choose $(u_1,v_1)$ and $(u_2,v_2)$ in a $\mathcal{W}$-ball of radius $n$. Then, using the Sobolev embedding $H^1\subset C^{0}$ and the boundedness of $W'$ and $W''$, we obtain
\begin{align}\label{i:304}
& \left\|\alpha((u_1,v_1))-\alpha((u_2,v_2))\right\|^2_{\mathcal{W}} \nonumber\\
& \quad = \left\|\{W'\ast u_1\}u_1-\{W'\ast u_2\}u_2\right\|_{L^2}^2+\left\|(\partial/\partial x)\left(\{W'\ast u_1\}u_1-\{W'\ast u_2\}u_2\right)\right\|_{L^2}^2\nonumber\\
& \quad \leq C\left\{\left\|\{W'\ast (u_1-u_2)\}u_1\right\|_{L^2}^2+\left\|\{W'\ast u_2\}(u_1-u_2)\right\|_{L^2}^2 + \left\|\{W''\ast (u_1-u_2)\}u_1\right\|_{L^2}^2 \right.\nonumber\\
& \quad \quad \left.+ \left\|\{W''\ast u_2\}(u_1-u_2)\right\|_{L^2}^2+ \left\|\{W'\ast (u_1-u_2)\}u'_1\right\|_{L^2}^2+\left\|\{W'\ast u_2\}(u'_1-u'_2)\right\|_{L^2}^2 \right\}\nonumber\\
& \quad \leq C(n,W)\left\|(u_1,v_1)-(u_2,v_2)\right\|^2_{\mathcal{W}},
\end{align}  
which is the local Lipschitz property for $\alpha$. Local boundedness is settled with an analogous computation. We apply~\cite[Theorem 4.5]{Tappe2012a} to deduce the existence of a unique local $\mathcal{W}$-valued mild solution $Z:=(\rho_Z,j_Z)$ to~\eqref{i:302} up to some $T>0$. Since the solution is c\`adl\`ag by~\cite[Remark 4.6]{Tappe2012a},  using the Sobolev embedding $H^1\subset C^{0}$, we can choose $T>0$ so that~\eqref{i:303} is satisfied.
\end{proof}
\begin{lemma} Let $X_0$ be a deterministic initial datum for~\eqref{i:301}. 
Then~\eqref{i:301} admits a unique local mild solution.
\end{lemma}
\begin{proof}
This follows from~\cite[Theorem 4.5]{Tappe2012a}, since (i) $A$ generates a $C_0$-semigroup of contractions on $\mathcal{W}$; (ii) $\alpha$ is locally Lipschitz and locally bounded, see Lemma \ref{il:7}; (iii) $B_{N,\delta}$ is locally Lipschitz and satisfies the linear growth condition, see~\cite[Lemma 4.5]{Cornalba2018aTR}; (iv) the noise $W_{\ep}$ is a $\mathcal{W}$-valued $Q$-Wiener process whose covariance operator $Q_{\sqrt{2}\ep}$ has rapidly decaying eigenvalues, see~\cite[Subsection 4.2]{Cornalba2018aTR}.
\end{proof}
Now let $X_{\ep}$ be the unique local mild solution to~\eqref{i:301}. For some positive constants $T,\delta$, and $k$, we define two relevant stopping times associated with~\eqref{i:301}, namely
\begin{align}\label{i:6001}
\tau_k & := \inf\left\{t>0:\|X_{\ep}(t)\|_{\mathcal{W}}\geq k\right\}\wedge T,\qquad
\mu_{\delta} := \tau_k\wedge \inf\left\{t>0:\min_{x\in\mathbb{T}}{\tilde{\rho}_{\ep}(x,t)}\leq \delta\right\}.
\end{align}
\begin{lemma}\label{il:8}
Fix $k>0$, $\delta>0$, and $T>0$. Let $X_{\ep}$ be the unique local mild solution to~\eqref{i:301}. The following statements hold:
\begin{description}
\item (a) The total mass of the system is conserved up to $\tau_k$, i.e., $\int_{\mathbb{T}}{\tilde{\rho}_{\ep}(x,s)\emph{\m} x}=\int_{\mathbb{T}}{\rho_{0}(x,s)\emph{\m} x}$ for all $s\leq \tau_k$.
\item (b) There exists a constant $C=C(X_0,W)$ such that, for all $x\in\mathbb{T}$ and for all $s\leq \mu_{\delta}$
\begin{align}\label{i:309}
-C\leq W'\ast\tilde{\rho}_{\ep}(x,s)\leq C,\quad -C\leq W''\ast\tilde{\rho}_{\ep}(x,s)\leq C.
\end{align}
\end{description}
\end{lemma}
\begin{proof}
(a) We consider the $\mathcal{W}$-inner product of the mild formulation of~\eqref{i:301} with the constant element $\zeta:=(1,0)\in \mathcal{D}(A^{\star})$, the symbol $^{\star}$ denoting the adjoint. As $A^{\star}\zeta=0$, we trivially get that 
\begin{align*}
\int_{0}^{T}{\mean{\int_{0}^{t}{\left\|\langle S(t-s)B_{N,\delta}(X_{\ep}(s)),A^{\star}\zeta\rangle\right\|^2_{L_2(\mathcal{W},\mathbb{R})}\m s}}\m t}<\infty.
\end{align*}
We define $\hat{\alpha}:=\alpha\circ R_k$, where 
\begin{align*}
R_k\colon\mathcal{W}\mapsto \mathcal{W}\colon y\mapsto
\left\{
     \begin{array}{ll}
      y, & \mbox{ if } \|y\|_{\mathcal{W}} \leq k, \\
     k\frac{y}{\|y\|_{\mathcal{W}}}, & \mbox{ if } \|y\|_{\mathcal{W}}>k
     \end{array}
     \right.
\end{align*} 
is a standard retraction map. Since the map $\hat{\alpha}$ is Lipschitz continuous, we have a unique global mild solution $\hat{X}_{\ep}$ to~\eqref{i:301} with $\alpha$ replaced by $\hat{\alpha}$, which then clearly satisfies $\mathbb{P}(\int_{0}^{T}{\|\hat{X}_{\ep}(t)\|_{\mathcal{W}}\,\m t}<\infty)=1$. Since we have predictability of both the deterministic and stochastic integrands involved in the definition of mild solution (to~\eqref{i:301} with $\alpha$ replaced by $\hat{\alpha}$), we follow the proof of~\cite[Proposition 2.10, part (ii)]{Frieler2001}, but \emph{only} with the specific choice of $\zeta$ made above (and \emph{not} with any $\zeta\in\mathcal{D}(A^{\star})$). We deduce that $\hat{X}_{\ep}$ satisfies, $\mathbb{P}\mbox{-a.s.}$
\begin{align*}
\langle \hat{X}_\ep,\zeta \rangle = \langle X_0,\zeta \rangle + \int_{0}^{t}{\left[\langle \hat{X}_{\ep}(s),A^{\star}\zeta \rangle + \langle \hat{\alpha}(\hat{X}_{\ep}(s)),\zeta \rangle\right]\m s} + \int_{0}^{t}{\langle B_{N,\delta}(\hat{X}_{\ep}(s)),\zeta\rangle\m W_{\ep}(s)}=\langle X_{0},\zeta\rangle.
\end{align*}
Uniqueness of mild solutions implies that $\hat{X_{\ep}}(s)=X_{\ep}(s)$ for all $s\leq \tau_k$, and the claim is settled. Notice that we have \emph{not} proved that $X_{\ep}$ is a weak solution to~\eqref{i:301}.

(b) The potential $W$ being smooth, there exists $C$ such that $-C\leq W'(y-x)\leq C$ for all $x,y\in\mathbb{T}$. If $s\leq \mu_{\delta}$, then $\tilde{\rho}_{\ep}(y,s)>0$ for every $y\in\mathbb{T}$. We deduce that $-C\tilde{\rho}_{\ep}(y,s)\leq W'(x-y)\tilde{\rho}_{\ep}(y,s)\leq C\tilde{\rho}_{\ep}(y,s)$, for all $y\in\mathbb{T}$. Since $\mu_{\delta}\leq \tau_{k}$, we can rely on (a) and integrate in $y$, thus deducing that $-C(X_0,W)\leq W'\ast\tilde{\rho}_{\ep}(x,s)\leq C(W,X_0)$ for all $x\in\mathbb{T}$ and for all $s\leq \mu_{\delta}$. An identical argument applies with $W''$ replacing $W'$.
\end{proof}
We now turn to the proof of our main existence and uniqueness result for~\eqref{i:300}. This result is an adapted version of~\cite[Proposition 4.10 and Theorem 1.4]{Cornalba2018aTR}.
\begin{theorem}[High-probability existence and uniqueness result]\label{ithm:10}
Fix $\nu\in(0,1)$, and fix $0<\delta<c_1<c_2<k$. Let $X_0=(\rho_0,j_0)\in\mathcal{W}$ be a deterministic initial condition, such that $\min_{x\in\mathbb{T}}{\rho_0(x)}> c_1$ and $\|X_0\|_{\mathcal{W}}< c_2$, and let $T>0$ be as in the statement of Lemma \ref{il:7}. Assume the scaling $N\ep^{\theta}=1$, for $\theta$ large enough. It is possible to choose a sufficiently large number of particles $N$ such that there exists a unique $\mathcal{W}$-valued mild solution $X_{\ep}=(\tilde{\rho}_{\ep},\tilde{j}_{\ep})$ satisfying~\eqref{i:300}, up to time $T$, on a set $F_{\nu}\in\mathcal{F}$ such that $\mathbb{P}(F_{\nu})\geq 1-\nu$. 
\end{theorem}
\begin{proof}
Consider the time $t\wedge \mu_{\delta}$, for $t\in[0,T]$, with $\mu_{\delta}$ defined in~\eqref{i:6001}. Let $X_{\ep}$ and $Z$ be the local mild solutions to~\eqref{i:301} and~\eqref{i:302}, respectively. We subtract the mild solution expressions for $X_{\ep}(t\wedge\mu_{\delta})$ and $Z(t\wedge\mu_{\delta})$, thus obtaining
\begin{align}\label{i:305}
X_{\ep}(t\wedge\mu_{\delta})-Z(t\wedge\mu_{\delta}) & =\int_{0}^{t\wedge\mu_{\delta}}{S(t\wedge \mu_{\delta}-s)\left[\alpha(X_{\ep}(s))-\alpha(Z(s))\right]\m s}\nonumber\\
& \quad +\int_{0}^{t\wedge\mu_{\delta}}{S(t\wedge \mu_{\delta}-s)B_{N,\delta}(X_{\ep}(s))\m W_{\ep}}.
\end{align}
We look for a small-noise regime estimate up to time $t\wedge\mu_{\delta}$. In order to do so, we first prove that 
\begin{align}\label{i:6002}
\|\alpha(X_{\ep}(s))-\alpha(Z(s))\|^2_{\mathcal{W}}\leq K^2_1\left(W,\|\rho_0\|_{H^1},T\right)\|X_{\ep}(s)-Z(s)\|^2_{\mathcal{W}},\qquad \mbox{for all }s\leq\mu_{\delta}.
\end{align} 
We reuse computation~\eqref{i:304} and deduce
\begin{align}\label{i:6000}
& \|\alpha(X_{\ep}(s))-\alpha(Z(s))\|^2_{\mathcal{W}} \nonumber\\
& \quad \leq 2\left\{\left\|\{W'\ast (\rho_{Z}-\tilde{\rho}_{\ep})\}\rho_{Z}\right\|_{L^2}^2+\left\|\{W'\ast \tilde{\rho}_{\ep}\}(\rho_{Z}-\tilde{\rho}_{\ep})\right\|_{L^2}^2 + \left\|\{W''\ast (\rho_{Z}-\tilde{\rho}_{\ep})\}\rho_{Z}\right\|_{L^2}^2 \right.\nonumber\\
& \quad \quad + \left.\left\|\{W''\ast \tilde{\rho}_{\ep}\}(\rho_{Z}-\tilde{\rho}_{\ep})\right\|_{L^2}^2+ \left\|\{W'\ast (\rho_{Z}-\tilde{\rho}_{\ep})\}\rho'_{Z}\right\|_{L^2}^2+\left\|\{W'\ast \tilde{\rho}_{\ep}\}(\rho'_{Z}-\tilde{\rho}'_{\ep})\right\|_{L^2}^2 \right\}\nonumber\\
& \quad =: T_1+\dots+T_6. 
\end{align}
For $s\leq \mu_{\delta}$, we bound the terms $T_2,T_4,T_6$ using Lemma \ref{il:8}, and we bound the terms $T_1,T_3,T_5$ using the Sobolev embedding $H^1\subset C^0$ and Lemma \ref{il:7}. Estimate~\eqref{i:6002} is proved.

We are now in the position to provide the small-noise regime estimate for~\eqref{i:305}. We closely follow the proof of~\cite[Proposition 4.10]{Cornalba2018aTR}. Let $q>2$. We use~\cite[Proposition 7.3]{Da-Prato2014a} to deduce that, for some $K_2=K_2\left(W,\|\rho_0\|_{H^1},T,q\right)$ and some $K_3=K_3(\sigma,\delta,T,q,k)$
\begin{align}
  & \mean{\sup_{s\in[0,t]}{\left\|X_{\ep}(s\wedge\mu_{\delta})-Z(s\wedge\mu_{\delta})\right\|^q_{\mathcal{W}}}} \nonumber\\
  & \quad \leq K_2\,\mean{\int_{0}^{t}{\left\|X_{\ep}(u)-Z(u)\right\|^q_{\mathcal{W}}\mathbf{1}_{[0,\mu_{\delta}]}(u)\m u}}\nonumber\\
  & \quad\quad +\mean{\sup_{s\in[0,T]}{\left\|\int_{0}^{s}{S(s\wedge\mu_{\delta}-u)B_{N,\delta}(X_{\ep}(u))\mathbf{1}_{[0,\mu_{\delta}]}(u)\m W_{\ep}}\right\|^q}}\nonumber\\ 
  & \quad \leq K_2\!\int_{0}^{t}{\mean{\sup_{s\in[0,u]}\left\|X_{\ep}(s\wedge\mu_{\delta})-Z(s\wedge\mu_{\delta})\right\|^q_{\mathcal{W}}}\m u}\nonumber\\
  & \quad\quad + K(\sigma,\delta,T,q)M^q(\ep,N)\mean{\int_{0}^{T}{\!\!(1+\|X_{\ep}(u)\|^q_{\mathcal{W}})\mathbf{1}_{[0,\mu_{\delta}]}(u)\m u}}\nonumber\\
  & \quad \leq K_2\!\int_{0}^{t}{\mean{\sup_{s\in[0,u]}\left\|X_{\ep}(s\wedge\mu_{\delta})-Z(s\wedge\mu_{\delta})\right\|^q_{\mathcal{W}}}\m u}
    +K_3M^q(\ep,N),\label{i:306}
\end{align}
where $M^q(\ep,N)$ was derived in~\cite[Lemma 4.5]{Cornalba2018aTR}, and decays to 0 as $\ep\rightarrow 0$ for $\theta$ large enough. It is easy to deduce that 
\begin{align}\label{i:307}
\mean{\sup_{s\in[0,T]}{\left\|X_{\ep}(s\wedge\mu_{\delta})-Z(s\wedge\mu_{\delta})\right\|^q_{\mathcal{W}}}}\leq K_3M^q(\ep,N)e^{TK_2}.
\end{align}
For some small enough $\eta>0$, define
$$
S:=\left\{\omega\in\Omega:\sup_{s\in[0,T]}{\left\|X_{\ep}(s\wedge\mu_{\delta})-Z(s\wedge\mu_{\delta})\right\|^q_{\mathcal{W}}}\leq \eta\right\}.
$$ 
Using the Chebyschev inequality in~\eqref{i:307}, we deduce that there exists $N$ large enough so that $\mathbb{P}(S)\geq 1-\nu$. If $\eta$ is chosen small enough, for any $\omega\in S$, we have that $\mu_\delta=\tau_k=T$. If this was not the case, we would have one of the following contradictions: on one hand, if $\mu_{\delta}<\tau_k\leq T$, since $\min_{x\in\mathbb{T}}{\rho_Z(x,s)}> c_1>\delta$ for all $s\in[0,T]$ thanks to Lemma \ref{il:7}, and since $\eta$ is small enough, we can use the embedding $H^1\subset C^0$ to deduce that $\min_{x\in\mathbb{T}}{\tilde{\rho}_{\ep}(x,\mu_{\delta})}>\delta$, contradicting the definition of $\mu_{\delta}$; on the other hand, if $\mu_{\delta}=\tau_k< T$, since $\|\rho_Z(s)\|_{\mathcal{W}}< c_2<k$ for all $s\in[0,T]$ thanks to Lemma \ref{il:7}, and since $\eta$ is small enough, we can use the same embedding $H^1\subset C^0$ to deduce that $\|\tilde{\rho}_{\ep}(\tau_k)\|_{\mathcal{W}}<k$, contradicting the definition of $\tau_k$. This concludes the proof.
\end{proof}
 \begin{rem}
 The main difference between this section and~\cite[Section 4]{Cornalba2018aTR} is the combination of a solution localisation via stopping times 
 (needed because the interaction term $\{W'\ast\tilde{\rho}_{\ep}(\cdot,t)\}\tilde{\rho}_{\ep}(\cdot,t)$ is superlinear) and the conservation of mass, see Theorem \ref{ithm:10} and Lemma \ref{il:8}.
 \end{rem}
\begin{rem}
The existence theory described in this subsection can be slightly simplified, as one could deduce the validity of~\eqref{i:309} for all $x\in\mathbb{T}$ and all $s\leq \tau_k$ (rather than for all $s\leq \mu_{\delta}$). In this case, the bounding constants would depend on $k$ (hence on $\|\rho_0\|_{H^1}$) rather than on $\int_{\mathbb{T}}{\rho_0(x)\m x}$, simply because of the embedding $H^1\subset C^0$.
The proof of Theorem \ref{ithm:10} could then be adapted by using the stopping time $\tau_k$ instead of $\mu_{\delta}$ in the small-noise regime analysis leading up to~\eqref{i:307}, thus making the use of Lemma \ref{il:8} superfluous. \\
However, Lemma \ref{il:8} provides a lower constant $K_2$ for the benefit of ~\eqref{i:307}. The reason for this can be deduced from~\eqref{i:6000}. The bounds associated with $T_1,\dots,T_6$ are of the type
\begin{align*}
T_i\leq C^2_i\|X_{\ep}(s)-Z(s)\|^2_{\mathcal{W}},\qquad i\in\{1,\dots,6\},
\end{align*}
where the constants $C_i$, $i\in\{1,\dots,6\}$, depend on $\|\rho_Z\|_{H^1}$ (or equivalently, on $\|\rho_0\|_{H^1}$ and $T$). However, the terms $T_2,T_4$, and $T_6$ can be controlled more precisely, as $C_2$, $C_4$, and $C_6$ can be computed with the initial mass $\int_{\mathbb{T}}{\rho_0(x)\m x}$ only (Lemma \ref{il:8}). In the case of an initial datum satisfying $\int_{\mathbb{T}}{\rho_0(x)\m x}\ll \|\rho_{0}\|_{H^1}$, this corresponds to obtaining a constant $K^2_1$ in~\eqref{i:6002} which is approximately half the one we would get if we did not rely on Lemma \ref{il:8} to deal with $T_2,T_4$, and $T_6$; this is simply because $K^2_1=C^2_1+\dots+C^2_6$, and $C^2_2+C^2_4+C^2_6$ would, in this case, be negligible compared to $C^2_1+C^2_3+C^2_5$. This is turn implies that the constant $K_2$ in~\eqref{i:307} can be scaled down by a factor up to $2^{q/2}$. Overall, this gives a smaller right-hand-side in~\eqref{i:307}, which reflects into a lower number of particles needed to meet the requirements of Theorem \ref{ithm:10}.
\end{rem}

\paragraph{Acknowledgements} All authors thank the anonymous referees for their meticulous reading of the manuscript and their valuable suggestions. 
FC is supported by a scholarship from the EPSRC Centre for Doctoral Training in Statistical Applied
Mathematics at Bath (SAMBa), under the project EP/L015684/1. JZ gratefully acknowledges funding by a Royal Society Wolfson Research Merit Award and the Leverhulme Trust for its support via grant RPG-2013-261.

\def\cprime{$'$} \def\cprime{$'$} \def\cprime{$'$}
  \def\polhk#1{\setbox0=\hbox{#1}{\ooalign{\hidewidth
  \lower1.5ex\hbox{`}\hidewidth\crcr\unhbox0}}} \def\cprime{$'$}
  \def\cprime{$'$}

\end{document}